\documentclass[10pt,a4paper,twoside,reqno]{amsart}

\usepackage{microtype}
\usepackage{amssymb}
\usepackage{mathtools}
\usepackage[shortlabels]{enumitem}
\usepackage[utf8]{inputenc}
\usepackage{color}
\usepackage{xspace}
\usepackage{fouridx}
\usepackage[mathcal]{euscript}			
\usepackage{tikz}
\usetikzlibrary{positioning,decorations.pathreplacing,decorations.markings,arrows.meta,calc,fit,quotes,cd,math,arrows,backgrounds,shapes.geometric}
\usepackage[a4paper,top=3cm,bottom=3cm,inner=2.5cm,outer=2.5cm]{geometry}
\usepackage{graphicx}
\definecolor{darkgreen}{rgb}{0,0.45,0}
\usepackage{bm}
\usepackage[colorlinks,citecolor=darkgreen,urlcolor=gray,final,hyperindex,linktoc=page,hyperfootnotes=true]{hyperref}
\usepackage[arrow, matrix, tips, curve, graph, rotate]{xy}
\SelectTips{cm}{10}

\usepackage{tabularx}

\usepackage{babel}
\usepackage{soul}

\setlist[enumerate]{label=(\roman*),itemsep=1ex,topsep=1ex}

\setlist[itemize]{itemsep=1ex,topsep=1ex}

\makeatletter

\renewcommand{\section}{\@startsection
{section}
{0}
{0mm}
{-\baselineskip}
{1\baselineskip}
{\centering \Large \bfseries}}

\renewcommand{\subsection}{\@startsection
{subsection}
{1}
{0mm}
{-\baselineskip}
{0.5\baselineskip}
{\large \bfseries}}

\makeatother
 
\usepackage[capitalize]{cleveref}
\crefname{equation}{}{}
\crefname{lem}{Lemma}{Lemmas}
\crefname{thm}{Theorem}{Theorems}
\crefname{defn}{Definition}{Definitions}
\crefname{conj}{Conjecture}{Conjectures}
\crefname{ex}{Example}{Examples}
\crefname{sec}{Section}{Sections}
\crefname{prop}{Proposition}{Propositions}
\crefname{rmk}{Remark}{Remarks}
\crefname{conv}{Convention}{Conventions}
\crefname{cor}{Corollary}{Corollaries}

\crefname{short-lem}{Lemma}{Lemmas}
\crefname{short-thm}{Theorem}{Theorems}
\crefname{short-defn}{Definition}{Definitions}
\crefname{short-conj}{Conjecture}{Conjectures}
\crefname{short-ex}{Example}{Examples}
\crefname{short-sec}{Section}{Sections}
\crefname{short-prop}{Proposition}{Propositions}
\crefname{short-rmk}{Remark}{Remarks}
\crefname{short-conv}{Convention}{Conventions}
\crefname{short-cor}{Corollary}{Corollaries}
\crefname{thmintro}{Theorem}{Theorems}

\theoremstyle{plain}
\newtheorem{thm}{Theorem}[subsection]
\newtheorem{cor}[thm]{Corollary}
\newtheorem{lem}[thm]{Lemma}
\newtheorem{prop}[thm]{Proposition}

\theoremstyle{remark}
\newtheorem{rmk}[thm]{Remark}

\theoremstyle{definition}
\newtheorem{defn}[thm]{Definition}
\newtheorem{ex}[thm]{Example}

\theoremstyle{plain}
\newtheorem{short-thm}{Theorem}[section]
\newtheorem{short-cor}[short-thm]{Corollary}
\newtheorem{short-lem}[short-thm]{Lemma}
\newtheorem{short-prop}[short-thm]{Proposition}
\newtheorem{short-conj}[short-thm]{Conjecture}
\theoremstyle{remark}
\newtheorem{short-rmk}[short-thm]{Remark}
\newtheorem{short-conv}[short-thm]{Convention}
\theoremstyle{definition}
\newtheorem{short-defn}[short-thm]{Definition}
\newtheorem{short-ex}[short-thm]{Example}

\newtheorem{thmintro}{Theorem}
\newenvironment{thm-intro}[1]
  {\renewcommand\thethmintro{#1}\thmintro\itshape}
  {\endthmintro}
\newenvironment{thm-introref}[2]
  {\renewcommand\thethmintro{#1}\thmintro(#2)\itshape}
  {\endthmintro}

\definecolor{mypurple}{rgb}{0.5, 0.0, 0.5}

\newcommand{\myemph}{\emph}
\newcommand{\ie}{i.e.\xspace}
\newcommand{\cf}{cf.\xspace}
\newcommand{\eg}{e.g.\xspace}

\newcommand{\defeq}{=_\mathrm{def}}
\newcommand{\co}{:}
\newcommand{\op}{^\mathrm{op}}
\newcommand\sfop{^\mathsf{op}}
\newcommand{\id}{\mathrm{id}}

\newcommand{\bimcomp}[1]{\underset{#1}\circ} 
\newcommand{\coend}[1]{\underset{#1}\otimes} 

\newcommand{\cat}{}
\newcommand{\catA}{\cat{A}}
\newcommand{\catB}{\cat{B}}
\newcommand{\catC}{\cat{C}}

\newcommand{\moncat}{}
\newcommand{\moncatM}{\moncat{M}}
\newcommand{\moncatN}{\moncat{N}}

\newcommand{\coc}{\mathit}
\newcommand{\cocC}{\coc{C}}

\newcommand{\opd}{\mathcal} 
\newcommand{\opdO}{\opd{O}}
\newcommand{\opdP}{\opd{P}}
\newcommand{\opdQ}{\opd{Q}}

\newcommand{\ccat}{\mathcal}
\newcommand{\ccatC}{\ccat{C}}
\newcommand{\ccatD}{\ccat{D}}

\newcommand{\rig}{}
\newcommand{\rigR}{\rig{R}}
\newcommand{\rigS}{\rig{S}}
\newcommand{\rigT}{\rig{T}}
\newcommand{\rigV}{\rig{V}}

\newcommand{\Bim}{\mathsf{Bim}}
\newcommand{\BimK}{\Bim(\ccatC)}

\newcommand{\EMK}{\mathsf{EMK}}

\newcommand{\emb}{\mathsf{y}}
\newcommand{\embA}{\emb_\catA}

\newcommand{\Psh}{\mathsf{Psh}}
\newcommand{\psh}[1]{\Psh{\left(#1\right)}}
\newcommand{\pshA}{\psh\catA}
\newcommand{\pshB}{\psh\catB}

\newcommand{\pshM}{\psh\moncatM}

\newcommand{\pshC}{\psh\cocC}

\newcommand{\Freesmc}{\mathsf{S}}
\newcommand{\freesmc}[1]{\Freesmc{\left(#1\right)}}
\newcommand{\freesmcA}{\freesmc\catA}
\newcommand{\freesmcB}{\freesmc\catB}
\newcommand{\freesmcC}{\freesmc\catC}

\newcommand{\freerig}[1]{\psh{\freesmc{#1}}}
\newcommand{\Freerig}{\freerig-}
\newcommand{\freerigA}{\freerig\catA}
\newcommand{\freerigB}{\freerig\catB}
\newcommand{\freerigC}{\freerig\catC}

\newcommand{\Sym}{\mathsf{Sym}}
\newcommand{\sym}[1]{\Sym{\left(#1\right)}}

\newcommand{\opdrig}[1]{\psh{\env{#1}}}
\newcommand{\opdrigP}{\opdrig\opdP}
\newcommand{\opdrigQ}{\opdrig\opdQ}

\newcommand\unit{\mathbf{1}}

\newcommand\set[1]{\underline{#1}}
\newcommand\symelt[2]{\angles{#1_1,\dots,#1_{#2}}}

\renewcommand\vec{\overline}
\renewcommand\hat{\widehat}

\newcommand\As{\mathrm{\mathcal{A}s}}
\newcommand\Asnu{\mathrm{\mathcal{A}s_{nu}}}
\newcommand\Com{\mathrm{\mathcal{C}om}}
\newcommand\Comnu{\mathrm{\mathcal{C}om_{nu}}}
\newcommand\Ezero{\mathcal{E}_0}

\newcommand\AsMod{\mathrm{\As\mathcal{M}od}}

\newcommand\rigmor[1]{\Phi{\left(#1\right)}}
\newcommand\monad[1]{\widehat{#1}}

\newcommand\slice[1]{_{/#1}}

\newcommand\sfU{\mathsf U}

\DeclareMathOperator*{\colim}{colim}
\DeclareMathOperator*{\bicolim}{bicolim}

\newcommand{\tto}{{\begin{tikzcd}[ampersand replacement=\&]{}\ar[r]\&{}\end{tikzcd}}}

\newcommand{\mto}{{\begin{tikzcd}[ampersand replacement=\&]{}\ar[r,mapsto]\&{}\end{tikzcd}}}

\newcommand{\stto}{{\begin{tikzcd}[ampersand replacement=\&, sep=small]{}\ar[r]\&{}\end{tikzcd}}}
\newcommand{\smto}{{\begin{tikzcd}[ampersand replacement=\&, sep=small]{}\ar[r,mapsto]\&{}\end{tikzcd}}}

\newcommand{\xto}[1]{\xrightarrow {#1}}            

\newcommand{\Env}{\mathsf{Env}}
\newcommand{\env}[1]{\Env{\left(#1\right)}}
\newcommand{\envP}{\env\opdP}
\newcommand{\envQ}{\env\opdQ}

\newcommand{\envC}{\env\catC}

\newcommand{\END}{\mathsf{End}}
\newcommand{\End}[1]{\END{\left(#1\right)}}
\newcommand{\EndM}{\End\moncatM}

\newcommand{\nc}{\mathsf}
\newcommand{\SET}{\nc{SET}}
\newcommand{\Set}{\nc{Set}}
\newcommand{\Fin}{\nc{Fin}}
\newcommand{\FinInj}{\nc{Fin_{Inj}}}
\newcommand{\FinSurj}{\nc{Fin_{Surj}}}
\newcommand{\FinBij}{\nc{Fin_{Bij}}}
\newcommand{\FinOrd}{\nc{Fin^{Ord}}}
\newcommand{\FinSurjOrd}{\nc{Fin^{Ord}_{Surj}}}

\newcommand{\Cat}{\nc{Cat}}
\newcommand{\Gpd}{\nc{Gpd}}

\newcommand{\CAT}{\nc{CAT}}

\newcommand{\TAME}{\nc{TAME}}

\newcommand{\COC}{\nc{COC}}
\newcommand{\FCOC}{\COC^{\mathsf{Free}}}
\newcommand{\FsetCOC}{\COC^{\mathsf{Free\,on\,Set}}}
\newcommand\Pres{\mathsf{Pres}}

\newcommand{\SIFTCAT}{\nc{SIFT}}
 \newcommand{\SIFTCATCOC}{\SIFTCAT\COC}

\newcommand{\SMCat}{\nc{SMCat}}
\newcommand{\SMCAT}{\nc{SMCAT}}

\newcommand{\RIG}{\nc{2}\text{-}\nc{RIG}}
\newcommand\Rig{\nc{2}\text{-}\nc{Rig}}

\newcommand\FRig{\Rig^{\nc{Free}}}

\newcommand\FsetRig{\Rig^\mathsf{Free\,on\,set}}
\newcommand\FgpdRig{\Rig^\mathsf{Free\,on\,gpd}}

\newcommand\CRig{\Rig^\mathsf{Conv}}

\newcommand\OpdRig{\Rig^{\nc{Opd}}}

\newcommand\VRig{\Rig_{\rigV}}
\newcommand\FVRig{\VRig^{\nc{Free}}}
\newcommand\OpdVRig{\VRig^{\nc{Opd}}}
\newcommand\FsetVRig{\VRig^\mathsf{Free\,on\,set}}

\newcommand{\SetSym}{\Set_{\nc{Sym}}}

\newcommand{\CatSym}{\nc{Cat}_{\nc{Sym}}}

\newcommand{\OPD}{\nc{OPD}}

\newcommand\Opd{\mathsf{Opd}}
\newcommand\Opdone{\mathsf{Opd}^{(1)}}

\newcommand\OpdBim{\Opd_{\mathsf{Bim}}}
\newcommand\OpdVone{\Opd_{\rigV}^{(1)}}
\newcommand\OpdV{\Opd_{\rigV}}

\newcommand\fun[2]{\left[#1,#2\right]}
\newcommand\Fun[2]{\big[#1,#2\big]}
\renewcommand\hom[3]{#1{\left(#2,#3\right)}}
\newcommand\Hom[3]{#1\big(#2,#3\big)}
\newcommand\HOM[3]{#1\Big(#2,#3\Big)}

\newcommand\covariant[2]{{#1}_{#2}}

\newcommand\contravariant[2]{{#1}^{#2}}
\newcommand\Scontravariant[2]{{#1}^{\vec #2}}
\newcommand\Scovariantotimes[2]{{(#1^\otimes)}_{\vec #2}}
\newcommand\bivariant[3]{{#1}^{#2}_{#3}}
\newcommand\Sbivariant[3]{{#1}^{\vec #2}_{#3}}
\newcommand\Sbivariantotimes[2]{{\left(#1^\otimes\right)}_{\vec #2}}
\newcommand\SSbivariantotimes[3]{{\left(#1^\otimes\right)}_{\vec #3}^{\vec #2}}
\newcommand\matrice[3]{#1{\left[#2\,;#3\right]}}
\newcommand\Smatrice[3]{#1{\left[\vec #2\,;#3\right]}}

\newcommand{\alg}[1]{{#1}\textsf{-}\nc{Alg}}
\newcommand{\AlgP}{\alg{\opdP}}

\newcommand\Alg[2]{{#2}\textsf{-}\nc{Alg}{\left(#1\right)}}
\newcommand{\Algname}[1]{\nc{Alg}{\left(#1\right)}}

\tikzset{tick/.style={postaction={decorate,decoration={markings,mark=at
position 0.5 with {\draw[-] (0,.4ex) -- (0,-.4ex);}}}}}
\tikzset{bigtick/.style={postaction={decorate,decoration={markings,mark=at
position 0.5 with {\draw[-] (0,.6ex) -- (0,-.6ex);}}}}}

\newcommand{\sbul}{\scriptstyle\bullet}
\tikzset{bul/.style={postaction={decoration={markings,mark=at position 0.5 with
{\node{$\sbul$};}},decorate}}}

\tikzset{Rightarrow/.style={double equal sign distance,>={Implies},->},
triple/.style={-,preaction={draw,Rightarrow}}}

\newcommand\postcompo[1]{#1\circ-}
\newcommand\precompo[1]{-\circ#1}
\newcommand\bicompo[2]{#1\circ-\circ#2}
\newcommand{\LMod}[4]{\hom{#1}{#2}{#3}^{\postcompo{#4}}}
\newcommand{\RMod}[4]{\hom{#1}{#2}{#4}^{\precompo{#3}}}
\newcommand{\LRMod}[5]{\hom{#1}{#2}{#4}^{\bicompo{#5}{#3}}}

\newcommand\bimssj[1]{\left|#1\right|}

\DeclarePairedDelimiter\angles\langle\rangle

\begin{document}
\title{Operadic 2-rigs}

\author[M.\,Anel]{Mathieu Anel} 
\address{Laboratoire J.-A. Dieudonn\'e, Universit\'e C\^ote d'Azur, Nice, France}
\email{mathieu.anel@protonmail.com}

\author[M.\,Fiore]{Marcelo Fiore}
\address{Department of Computer Science and Technology, University of Cambridge, United Kingdom}
\email{marcelo.fiore@cl.cam.ac.uk}

\author[N.\,Gambino]{Nicola Gambino}
\address{Department of Mathematics, The University of Manchester}
\email{nicola.gambino@manchester.ac.uk}

\keywords{Symmetric 2-rigs, operads, bimodules.}

\subjclass[2020]{18M60, 18N10,18M05, 18D60}

\begin{abstract}
We show that the bicategory of operads and bimodules can be embedded into the bicategory of symmetric 2-rigs, a categorification of commutative rings.
In order to do this, we introduce the notion of an operadic 2-rig and show that the full sub-bicategory of symmetric 2-rigs spanned by operadic 2-rigs has the universal property of being a completion under Eilenberg--Moore--Kleisli objects.
\end{abstract}

\date{\today}

\maketitle

\setcounter{tocdepth}{1}
\tableofcontents

\section{Introduction}

\subsection{Context and motivation}

This paper seeks to advance the theory of operads, which has its origins in algebraic topology~\cite{MayJ:geoils,BoardmanJ:homias}, but has since found applications also in algebra~\cite{LodayJL:algo,LivernetM:frolma} and beyond~\cite{MarklM:opeatp,LurieJ:higa}.
Operads capture in a uniform way various kinds of algebraic structures. Indeed, each operad has an associated category of algebras; for example, there are operads $\As$ and $\Com$, whose categories of algebras are
associative and commutative algebras, respectively. For some applications, it is necessary to consider coloured operads (also known as symmetric multicategories)~\cite{yau2016colored}, which generalise operads by allowing many-sorted algebraic structures, such as that of an algebra and a module over it. 
Another example is the coloured operad whose algebras are the operads themselves~\cite{BergerC:rescorha,LeinsterT:higohc}.
Below, for simplicity, we shall simply speak of operads rather than coloured operads.

Operads are algebraic structures in their own right, with a corresponding notion of morphism, which captures many---but not all---universal constructions  between algebras. For example there is a morphism of operads $U \co \As \to \Com$ which determines the functor mapping  a commutative algebra to its underlying associative algebra.
The ``restriction'' functor $U$ has also a left adjoint ``induction'' functor $F$, sending an associative algebra to its Abelianisation.
This functor 
cannot in general be described as the restriction along a morphism $\Com \to \As$.
More generally, the functor $T \co \alg\As \to \alg\As$ taking an associative algebra $A$ to its tensor algebra $T(A)$ is not given by ``restriction'' or ``induction'' along a morphism of operads from $\As$ to itself.
This raises the problem of defining additional maps between operads, so that more functors between their algebras can be represented in a simple algebraic form.

Building on the notion of a module for an operad~\cite{KapranovM:modmto,RezkC:spaasc}, it is possible
to define operad bimodules and show that they  
provide additional maps~\cite{FresseB:modoof,GambinoN:opebaf}. Indeed, operad bimodules induce functors between categories of algebras for operads,
which include the tensor algebra functor mentioned above, as well as induction and restriction functors~\cite{FresseB:modoof}. Such functors were called analytic in~\cite[Section~4.5]{GambinoN:opebaf} because they generalise the analytic functors between categories of presheaves of~\cite{FioreM:carcbg}, which in turn extend 
the analytic functors originally introduced by Joyal in~\cite{JoyalA:fonaes}. Operads, bimodules, and bimodule morphisms 
give rise to a bicategory $\OpdBim$ \cite[Section~4.4]{GambinoN:opebaf},
just as rings, ring bimodules, and bimodule morphisms do.

The original motivation for this paper was to provide a more conceptual understanding of operad bimodules and of the analytic functors between categories of algebras for  operads induced by them.
This, in turn, aims to make possible applications to cartesian closed and differential structure in bicategories of interest, as discussed below.
Our results in this paper therefore contribute also to a line of research begun in~\cite{FioreM:carcbg}, where the bicategory $\CatSym$ of symmetric sequences was defined by extending Joyal's work in enumerative combinatorics~\cite{JoyalA:thecsf,JoyalA:fonaes} and categorical approaches to operads~\cite{BaezJ:higda,KellyG:opemay}. The bicategories $\CatSym$ and its extension $\OpdBim$ have been shown to possess a rich structure~\cite{FioreM:monbdll,GambinoN:monkbap,GambinoN:unitctp}, making them relevant in mathematical logic~\cite{HylandM:somrgd}, theoretical computer science~\cite{FioreM:stapss,GalalZ:fixo2c,Olimpieri,OngL:quaslc}, and operad theory~\cite{DotsenkoV:endpbt,PavlovD:admrcs}. We therefore hope that our results may be of wide interest.

\subsection{Main results}

The starting point for our work are \emph{(symmetric) 2-rigs}, which are a categorification of commutative rigs, \ie commutative rings without negatives. 
In the same way that a commutative ring is a set equipped with
the structures of an Abelian group and of a commutative monoid interacting by means of a distributivity law,
a symmetric 2-rig is a category with small colimits and with a symmetric monoidal structure interacting in the sense of the tensor product preserving colimits in each variable.
For example, the category of sets, viewed as a cartesian monoidal category, 
and the category of vector spaces with the tensor product,
are symmetric 2-rigs.
See~\cite{Baez:schfcp,BrandenburgM:refdca,ChirvasituA:funpga,LoregianF:dif2r} for some recent work on 2-rigs.

Algebras for operads are often considered within a 2-rig. 
Indeed, a 2-rig $\rigR$ defines an  operad $\End\rigR$, called the `endomorphism'  operad of~$R$. With this definition, an algebra for an operad $\opdP$ in $\rigR$ is defined as an operad morphism 
from $\opdP$ to $\End\rigR$. This definition can be reformulated using the fact that an operad $\opdP$ determines a  2-rig $\opdrigP$, constructed from $\opdP$ in two steps. First, one considers 
the enveloping symmetric monoidal category $\envP$ of $\opdP$, given by the biadjunction between  operads and symmetric monoidal categories~\cite{ElmendorfA:percma}. Secondly, one takes the category of  presheaves $\opdrigP$ over $\envP$, considered as a 2-rig with respect to the convolution monoidal structure~\cite{DayB:clocf}. With these definitions,  operad morphisms $\opdP \to \End\rigR$ are equivalent to a 2-rig morphism $\opdrigP\to \rigR$. 

We define \emph{operadic $2$-rigs} to be the 2-rigs  of this form, \ie~those in the essential image of the pseudofunctor
\begin{equation}
\label{eq:from-smultcat-to-2rigs}
\tag{$\ast$}
\begin{tikzcd}[column sep = large]
\Opd \ar[r, "\Psh \circ \Env"]  & \Rig \mathrlap{,}
\end{tikzcd}
\end{equation}
where $\Opd$ is the bicategory of  operads, morphisms, and transformations, and $\Rig$ is the bicategory of 2-rigs, their morphisms (\ie symmetric strong monoidal cocontinuous functors) and transformations
(\ie symmetric monoidal transformations). The bicategory of operadic 2-rigs is then defined as the full
sub-bicategory of the bicategory $\Rig$ of  2-rigs spanned by  operadic 2-rigs, thus fitting
into a factorisation of the pseudofunctor in~\eqref{eq:from-smultcat-to-2rigs}  as an essentially surjective functor followed by a fully faithful one, as in
\[
\begin{tikzcd} 
\Opd  \ar[r, two heads]  & \OpdRig \ar[r, hook] &    \Rig \mathrlap{.} 
  \end{tikzcd}
  \]
 The pseudofunctor in~\eqref{eq:from-smultcat-to-2rigs} is faithful (\cref{lem:opd-in-opdbim}) but not full, as there are more 2-rig morphisms between operadic 2-rigs than morphisms between operads.
 It is therefore natural to ask whether it is possible to give a characterisation of these 2-rig morphisms purely in terms of  operads. 
  Pleasingly, our first main result shows that these are exactly the operad bimodules discussed above.

\begin{thm-introref}{A}{see \cref{thm:opdrig=opdbim}}
\label{thmA}
The bicategory $\OpdRig$ of operadic 2-rigs is biequivalent to the bicategory $\OpdBim$ of  operads and their bimodules.
\end{thm-introref}

Read a different way, this theorem says that the bicategory of  operads and their bimodules is equivalent to a full sub-bicategory of the bicategory of 2-rigs, which provides the desired alternative understanding of operad bimodules.

Our second main result provides a characterisation of the bicategory of operadic 2-rigs by a universal property, which is the main technical step to prove \cref{thmA} and provides
a reinterpretation of~\cite[Theorem~5.4.5]{GambinoN:opebaf}  in terms of  2-rigs.

\begin{thm-introref}{B}{see \cref{thm:opdrig=em-frig}}
\label{thmB}
The bicategory $\OpdRig$ of operadic 2-rigs is the closure of the bicategory of free symmetric 2-rigs under the construction of Eilenberg--Moore objects.
\end{thm-introref}

Here, free symmetric 2-rigs are those of the form $\freerigA$, for a category $\catA$, where $\freesmcA$ is the free symmetric monoidal category on $A$.
With these results in place, we can provide a conceptually clear account of the analytic functors between categories of algebras for operads induced by operad bimodules.
We have seen that, for an operad $\opdP$, the category $\Alg\rigR\opdP$ of $\opdP$-algebras in a 2-rig $\rigR$ can be defined 
as the category of morphisms of operads $\opdP\to \End\rigR$ 
or equivalently as the category of morphisms of 2-rigs $\opdrigP\to \rigR$.
This second description factors the functor $\opdP\mapsto \Alg\rigR\opdP$ through the bicategory of operadic 2-rigs
\[
\begin{tikzcd}[column sep = large]
\Algname\rigR:\Opd\sfop \ar[rr, "\Psh \circ \Env", two heads] &&  \big(\OpdRig\big)\sfop \ar[rr, "\hom\Rig-\rigR"] && \CAT \mathrlap{.}
\end{tikzcd}
\]
This factorisation show that categories $\Alg\rigR\opdP$ are not only functorial with respect to operad morphisms, but also with respect to 2-rig morphisms.
This is somehow the motivation to replace operads by their operadic 2-rigs.
Then, analytic functors between categories of algebras for operads, as defined in~\cite[Section~4.5]{GambinoN:opebaf}, arise by composing the pseudofunctors
\[
\begin{tikzcd}[column sep = large]
(\OpdBim)\sfop \simeq \big(\OpdRig\big)\sfop \ar[rr, "\hom\Rig-\rigR"] && \CAT \,.
\end{tikzcd}
\]
Examples are given in \cref{ex:adj-induction-restriction,ex:free-alg}.
As we show as well in \cref{sec:enror}, these results extend to the enriched setting in the expected way. 

As a consequence of our main results, we also obtain that $\OpdBim$ is
biequivalent to the Kleisli bicategory for a relative pseudomonad, answering a question of Richard Garner.
Along the way, we prove other results that may be of independent interest. Indeed, 
the embedding of  operads and their bimodules into symmetric 2-rigs relies on the
analysis of Eilenberg--Moore and Kleisli objects in a special class of bicategories, called tame (\cref{defn:tame}),
where these objects have the special property of coinciding. Here,
we show that the bicategory of symmetric 2-rigs is tame and Eilenberg--Moore--Kleisli
complete (\cref{cor:RIG-is-tame,thm:rig-em-complete}).

In~\cite{AnelM:sym2rcc} we already apply the results in this paper to provide a unified account of the cartesian closed
structure of $(\CatSym)\sfop$ and $(\OpdBim)\sfop$ established in~\cite{FioreM:carcbg} and \cite{GambinoN:opebaf}, respectively,
via a characterisation of coexponentiable symmetric 2-rigs. Furthermore, thanks to possibility of developing
a counterpart of the theory of K\"ahler differentials for symmetric 2-rigs, 
 the embedding of operad bimodules into symmetric 2-rigs obtained here allows us to develop a theory of differentials for 
operad bimodules that extends the one for categorical symmetric sequences
in~\cite{FioreM:monbdll}, which we will present in a future paper.

Finally, we expect that our results can be extended to the $(\infty,2)$-category of 
$\infty$-operads~\cite{HeutsG:simdht,LurieJ:higa}.
Steps in this direction have been taken by Zetto \cite{Zetto} where his `operadic presheaf category’ are a higher analog of our `operadic 2-rigs’. 
See  \cref{rmk:Getzler,rmk:enriched-opd-morphism} for additional information on the relationship with his work.

\subsection{Outline of the paper}

We begin in \cref{sec:background} by reviewing the background of the theory of bicategories needed
in the rest of the paper. \cref{sec:sym2rig}  establishes some auxiliary results on the bicategories $\COC$
of cocomplete categories and~$\RIG$ of symmetric 2-rigs. In particular, we establish
that they are pseudomonadic over the bicategory of categories in two
(equivalent) ways, and
that they are tame and Eilenberg--Moore-Kleisli complete. We
present and study operadic 2-rigs and our main results in~\cref{sec:operadic-2-rigs}. 
We conclude the main development in \cref{sec:enror} by treating the enriched case.
\cref{app:anasf} introduces some notation to deal with analytic functors.

\subsection*{Acknowledgements}
We would like to thank Andr\'e Joyal, who first suggested the link of  operads with symmetric 2-rigs, Clemens Berger for useful conversations,
and Markus Zetto for bringing his paper~\cite{Zetto} to our attention.
Mathieu Anel acknowledges that the research leading to these results has received funding from the European Research Council (ERC) under the European Union's Ninth Framework Programme Horizon Europe (ERC Synergy Project Malinca, Grant Agreement n.~101167526). 
Marcelo Fiore acknowledges that this material is
based upon work supported by EPSRC via grant EP/V002309/1.
Nicola Gambino acknowledges that this material is based upon work supported by
the US Air Force Office for Scientific Research under award number
FA9550-21-1-0007,  by EPSRC via grant EP/V002325/2, 
and ARIA via grant MSAI-PR01-P12.

\section{Background} 
\label{sec:background}

\subsection{Conventions}

We adopt the same size conventions as in~\cite{AnelM:sym2rcc}. In particular, we fix two regular cardinals $\kappa$ and $\lambda$, with $\aleph_0 < \kappa < \lambda$. By definition, 
sets of cardinality less than $\kappa$ will be called small sets, sets of cardinality less than $\lambda$ will be called sets, sets of any other cardinality will be called large sets.
The notions of small category, category, and large category are defined accordingly. By a locally small category we mean a category whose hom-sets are small sets. We write $\Set$ 
for the category of small sets and $\SET$ for the large category of sets. Similarly, we write $\Cat$ for the category of small categories, and $\CAT$ for the large category of categories.
Analogous size conventions will be made about  operads, although in that context we shall concentrate almost exclusively on small  operads.

For categories $\catA$ and $\catB$, we write~$\catA \times \catB$ for their product and~$\fun\catA\catB$ for their exponential, which is the category of functors from $A$ to $B$ and natural transformations. 
For a category $\catA$, we write $\catA\op$ for its opposite. 
When $\catA$ is small the category of \emph{presheaves} over $\catA$ is $\pshA \defeq \fun{\catA\op}\Set$. 
This is a locally small category, but it is not small in general. 
The Yoneda embedding is written $\embA \co \catA \to \fun{\catA\op}\Set$.
In the following, we shall often omit mention of the Yoneda functor and treat it as an inclusion.

\subsection{Bicategories}

We shall assume familiarity with the fundamental notions and results of two-dimensional category theory referrring to~\cite{LackS:a2cc,JohnsonN:twodc} for background.
This paper is written entirely in terms of the theory of bicategories, \ie~weak 2-categories, in order to
ensure that all our results are invariant under the appropriate notion of equivalence. 
Hence, we will speak of bicategories, pseudofunctors (also known as homomorphisms), pseudonatural transformations, and modifications. Accordingly, we will have notions of biadjunction,
bilimit and bicolimit, and pseudomonad. 
While some of
the structures that we deal with are sometimes presented in a stricter way, we will not keep track of this
to avoid proliferation of terminology, trusting expert readers to be able to recognise situations of this
kind.  

For two objects $A$, $B$ in a bicategory $\ccatC$, we write $\hom\ccatC AB$ for the hom-category of maps from $A$ to~$B$ and 2-cells between them.
For maps $f \co A \to B$ and $g \co B \to C$, we write $g \circ f \co A \to C$ for their horizontal composite.
For 2-cells $\alpha \co f \Rightarrow g$ and $\beta \co g \Rightarrow h$, we write $\beta \cdot \alpha \co f \Rightarrow g$ for their vertical composition. 
We say that a map $f \co A \to B$ is an \emph{equivalence} if there exists a map $g \co B \to A$ and isomorphisms $\eta \co \id_A \Rightarrow g \circ f$, $\varepsilon \co f \circ g \Rightarrow \id_B$.
When this happens, we say that $A$ and $B$ are \emph{equivalent} and write $A \simeq B$.
An \emph{adjunction} in $\ccatC$ consists of a pair of maps $f \co A \to B$ and $g \co B \to A$ together with 2-cells $\eta \co \id_A \Rightarrow g \circ f$, $\varepsilon \co f \circ g \Rightarrow \id_B$ satisfying the usual triangular laws. 
By default, our bicategories will be large, \ie~have  a large set of equivalence classes of objects and large hom categories.

Recall from~\cite{StreetR:fibb} that a pseudofunctor $G \co \ccatD \to \ccatC$ has a left biadjoint if and only if 
for each $A \in \ccatC$, we have an object $FA$ in $\ccatD$ and a map $\eta_A \co A \to GFA$ in $\ccatD$ which are universal in the sense that composition with $\eta_A$ induces an equivalence of categories
\[
\begin{tikzcd}[column sep = large]
\hom\ccatD{FA}X  \ar[r, "(-) \circ \eta_a"] & \hom\ccatC A{GX} 
\end{tikzcd}
\]
for every $X$ in $\ccatD$.
Analogously to the 1-categorical setting, right biadjoints preserve bilimits and left biadjoints preserve bicolimits. 
We write $\ccatC \simeq \ccatD$ to indicate that two bicategories $\ccatC$ and $\ccatD$ are biequivalent.
A pseudofunctor $F \co \ccatC \to \ccatD$ is a biequivalence if and only if it is
\emph{essentially surjective} (\ie~for every $X$ in $\ccatD$ there exists $a$ in $\ccatC$ such that $FA \simeq X$) 
and \emph{fully faithful} (\ie~for every $A$ and $B$ in $\ccatC$, the functor 
$F_{A,B} \co \hom\ccatC AB \to \hom\ccatC{FA}{FB}$
is an equivalence of categories). 

We will frequently use that, for a pseudofunctor $F \co \ccatC \to \ccatD$,
there is an \emph{image factorisation} 
\[
\begin{tikzcd}
\ccatC \ar[dr, "F"'] \ar[r, "L" , two heads] & \mathsf{Im}(F) \ar[d, hook, "R"]  \\
 & \ccatD \,,
 \end{tikzcd}
\]
where $\mathsf{Im}(F)$, called the \emph{image} of $F$, is defined as the full 
sub-bicategory of $\ccatD$ spanned by the objects in the essential image
of~$F$, \ie~the objects of $\ccatD$ equivalent to one of the form $F(a)$ for some $a$ in $\ccatC$. The pseudofunctor $L$ is essentially surjective and the
pseudofunctor $R$ is fully faithful.
Let us note that the image factorisation is slightly different from the Gabriel factorisation considered in~\cite{GambinoN:opebaf}, 
since a Gabriel factorisation involves asking the pseudofunctor $L$ to be bijective on objects and hence
it is not invariant under equivalence. For our purposes, it will be useful that $\mathsf{Im}(F)$ is closed under equivalences.

\subsection{Monads in bicategories} 

We review
some basic concepts of the formal theory of
monads~\cite{StreetR:fortm,LackS:fortmII}. Let $\ccatC$ be a bicategory. 
Recall that, for $A$ in $\ccatC$, the hom-category $\hom\ccatC AA$ admits a monoidal structure given by composition, with unit the identity map on $A$. The next definition is stated explicitly in order to make the relation with
\cref{thm:2rigmonad} clear.

\begin{defn} For $A$ in $\ccatC$, a \emph{monad} on $A$ is a monoid in $\hom\ccatC AA$, \ie~a map $p \co A \to A$ equipped with a 2-cell~$\mu \co p \circ p \Rightarrow p$, called the \emph{multiplication}, and a 2-cell 
$\eta \co \id_A \Rightarrow p$, called the \emph{unit}, satisfying the usual associativity and unitality axioms. 
\end{defn}

We shall refer to a monad as a pair $(A,p)$, leaving its multiplication and unit implicit. 

For a monad $(A,p)$ in $\ccatC$ and an object $X$ in $\ccatC$, the endomorphism $p \co A\to A$ induces  
two endofunctors $\postcompo p\co\hom\ccatC XA$ and $\precompo p\co\hom\ccatC AX$ which are monads in $\CAT$.
A \emph{left $p$-module} is a pair $(X,\ell)$ where $X$ is an object in $\ccatC$ and $\ell \co X\to A$ is an algebra
for the monad $\postcompo p$ on $\hom\ccatC XA$.
Explicitly, a left $p$-module is a map $\ell \co X \to A$ equipped with a 2-cell $\lambda \co p \circ\ell \Rightarrow \ell$ satisfying the axioms for a left action of $p$.
A \emph{right $p$-module} is a pair $(X,r)$ where $X$ is an object in $\ccatC$ and $r \co A\to X$ is an algebra for the monad $\precompo p$ on $\hom\ccatC AX$.
Explicitly, a right $p$-module is a map $r \co A \to X$
  equipped with a 2-cell $\rho \co r \circ p \Rightarrow r$ satisfying the axioms for a right action of $p$.
The object $X$ is called the underlying object of the module.
We denote by $\LMod\ccatC XAp$ and $\RMod\ccatC ApX$ the categories of left and right modules with underlying object $X$.
The definitions of these categories are natural in $X$ and define pseudofunctors on $\ccatC\op$ and $\ccatC$, respectively.

\begin{defn} \label{defn:em-kleisli-object}
Let $p \co A \to A$ be a monad in $\ccatC$. 
\begin{itemize}
\item An \emph{Eilenberg--Moore object} for $p$ is an object $A^p$ and a left $p$-module $u \co A^p \to A$ which is universal, in the sense that composition with $u$ induces an equivalence of categories
\[
\hom\ccatC X{A^p} \stto \LMod\ccatC XAp \,.
\]
\item\label{defn:em-kleisli-object:kleisli} A \myemph{Kleisli object} for $p$ is an object $A_p$ and a right $p$-module $f \co A \to A_p$ which is universal, in the sense that composition with $f$ induces an equivalence of categories
\[
\hom\ccatC {A_p}X \stto \RMod\ccatC ApX \,.
\]
\end{itemize}
\end{defn}

Explicitly, an Eilenberg--Moore object as above has the universal property that for every left $p$-module $\ell \co X \to A$ we have an essentially unique map $\ell^p \co X \to A^p$ making the following diagram commute up to unique isomorphism:
\[
\begin{tikzcd}
& A^p \ar[d, "u"]\\
X \ar[r,"\ell"'] \ar[ru,dotted ,"\ell^p"]& A  \,.
\end{tikzcd}
\]
By choosing $\ell \co X\to A$ to be $p \co A\to A$ with the canonical  left action of $p$ given by the monad multiplication, one can show that the map $p^p \co A\to A^p$ is left adjoint to $u$ in $\ccatC$ \cite[Theorem 2]{StreetR:fortm}. 
Similarly, a Kleisli object has the universal property that for every right $p$-module $r \co A \to X$ we have an essentially unique map $r_p \co A_p \to X$ making the following diagram commute up to unique isomorphism:
\[
\begin{tikzcd}
A_p \ar[dr, dotted, "r_p"] & \\
A \ar[r, "r"'] \ar[u, "f"] & X  \,.
\end{tikzcd}
\]
Choosing $r \co A\to X$ to be $p \co A\to A$ with the canonical right action of $p$, the map $p_p \co A_p\to A$ is a right adjoint to $f$ in $\ccatC$ \cite[Theorem 2]{StreetR:fortm}.
It will be useful to recall that Eilenberg--Moore objects are a lax bicolimit and Kleisli objects are a lax limit~\cite{StreetR:fortm}. Hence, the former are preserved by left biadjoints
and the latter are preserved by right biadjoints.

Eilenberg--Moore and Kleisli objects for a monad in $\CAT$ are given by the category of Eilenberg--Moore algebras and the Kleisli category of the monad, respectively.
In other 2-categories the situation may be quite different, as we shall see.

\begin{defn} \label{defn:bimodule-in-k} 
Let $p \co A \to A$ and $q \co B \to B$ be two monads in $\ccatC$. 
A \emph{$(q,p)$-bimodule} is a triple 
$(f, \lambda, \rho)$ 
of a map $f \co A \to B$ equipped with the structure of a left $q$-module $\lambda \co q \circ f \Rightarrow f$ and a right $p$-module 
$\rho \co f \circ p \Rightarrow f$ which commute with each other, in the sense that the diagram
\[
\begin{tikzcd} 
q \circ f \circ p \ar[r, "f \circ \rho"] \ar[d, "\lambda \circ f"'] & q \circ f \ar[d, "\lambda"] \\
f \circ p \ar[r, "\rho"'] & f 
\end{tikzcd}
\]
commutes.
\end{defn}

There is then an evident category $\LRMod\ccatC ApBq$ of $(q,p)$-bimodules and bimodules morphisms.

\begin{rmk}
Our notation $\RMod\ccatC ApX$, $\LMod\ccatC XBq$, and $\LRMod\ccatC ApBq$ for the categories of 
left modules, right modules, and bimodules are consistent with the superscript notation of Eilenberg--Moore objects,
since these categories are Eillenberg--Moore objects in $\CAT$.
\end{rmk}

We record the following result for future reference.
\begin{lem}
\label{lem:bim=k2em}
If $\ccatC$ has Eilenberg--Moore and Kleisli objects, the universal property of these objects gives an equivalence of categories
\[
\hom\ccatC{A_p}{B^q} \simeq \LRMod\ccatC ApBq \,. 
\]
\end{lem}

\begin{rmk}
\label{rmk:k2em}
For any monad $(A,p)$, the endomorphism $p \co A\to A$ defines a $(p,p)$-bimodule which, under the equivalence of \cref{lem:bim=k2em},  corresponds to a canonical morphism $i_A \co A_p\to A^p$ such that the composition $A\to A_p\to A^p\to A$ is equivalent to $p$.
This provides a commutative diagram
\[
\begin{tikzcd}
A_p\ar[rr,"i_A"]&& A^p\ar[d,"u"]
\\
A\ar[rr,"p"'] \ar[u,"f"] \ar[rru,"{p^p}" description, near start]
&& A\ar[from=llu,"{p_p}" description, near end, crossing over]
\,.
\end{tikzcd}
\]
\end{rmk}

\subsection{Monads in tame bicategories}
\label{sec:monad-tame}

Here, we are interested in Eilenberg--Moore and Kleisli objects in
bicategories whose hom-categories possess particular classes of colimits and
whose composition functors preserve them in each variable. Let us recall the
notion of a tame bicategory~\cite[Definition~4.2.1]{GambinoN:opebaf}.

\begin{defn}  \label{defn:tame}
\leavevmode
\begin{itemize} 
\item A bicategory $\ccatC$ is said to be \emph{tame} if it has local reflexive coequalisers, \ie~for every $A$ and $B$ in $\ccatC$, the hom-category $\hom\ccatC AB$
has reflexive coequalisers and the composition functors of $\ccatC$ preserve reflexive coequalizers in each variable.
\item A pseudofunctor $F \co \ccatC \to \ccatD$ between tame bicategories is said to be \emph{tame} if, for every $A$ and $B$ in $\ccatC$, the
functor $F_{A,B} \co \hom\ccatC AB \to \hom\ccatD{FA}{FB}$ preserves local reflexive coequalizers.
\end{itemize} 
\end{defn}

\begin{rmk}
\label{rmk:sub-tame=tame}
Since the definition of tameness depends only on the hom categories, any full sub-bicategory of a tame bicategory is tame.
\end{rmk}

Examples of tame categories will be given in \cref{lem:tame-coc,cor:RIG-is-tame}.

\medskip
One remarkable aspect of tame bicategories is that
Eilenberg--Moore and Kleisli objects in them coincide, 
as shown in \cite[Corollary~5.2.12]{GambinoN:opebaf}. 

\begin{lem}
\label{lem:em=k}
Given a monad $(A,p)$ in a tame bicategory:
\begin{enumerate}
\item If $f \co A\to A_p$ is a Kleisli object, then its right adjoint $p_p \co A_p\to A$ is an Eilenberg--Moore object for $p$.
\item If $u \co A^p\to A$ is an Eilenberg--Moore object, then its left adjoint $p^p \co  A\to A^p$ is a Kleisli object for $p$.
\item If both Eilenberg--Moore and Kleisli objects of $p$ exist, then the canonical morphism $i_A \co A_p\to A^p$ of \cref{rmk:k2em} is invertible.
\end{enumerate}
\end{lem}

\begin{defn} A tame bicategory $\ccatC$ is said to be \emph{Eilenberg--Moore--Kleisli complete} if every monad in~$\ccatC$ has an Eilenberg--Moore (and hence a Kleisli) object.
\end{defn} 

Also recall that a tame pseudofunctor between tame bicategories preserves Eilenberg--Moore objects~\cite[Proposition~5.2.13]{GambinoN:opebaf}. For tame bicategories $\ccatC$ and $\ccatD$, we write $\hom\TAME\ccatC\ccatD$ for the
full sub-bicategory of the hom-bicategory~$\fun\ccatC\ccatD$ spanned by tame pseudofunctors.

For a tame bicategory $\ccatC$, its completion under Eilenberg--Moore objects as a tame bicategory is a tame bicategory $\EMK(\ccatC)$ with a tame pseudofunctor $J \co \ccatC \to \EMK(\ccatC)$ 
which is universal, in the sense that for every tame bicategory with Eilenberg--Moore objects $\ccatD$,
composition with $J$ induces an equivalence of bicategories
\[
\begin{tikzcd}[column sep = large]
\hom\TAME{\EMK(\ccatC)}\ccatD \ar[r, "(-) \circ J"]  & \hom\TAME\ccatC\ccatD  \,.
\end{tikzcd}
\]
Since Eilenberg--Moore objects and Kleisli objects coincide in tame bicategories,
$\EMK(\ccatC)$ is also the completion of $\ccatC$ under Kleisli objects as a tame bicategory. This
is not to be confused with the Eilenberg--Moore completion of $\ccatC$ as a mere bicategory,
as studied by Lack and Street in \cite{LackS:fortmII}, which has a different universal property, and does not coincide with $\EMK(\ccatC)$ in general.

As shown in~\cite[Section~5.4]{GambinoN:opebaf}, the Eilenberg--Moore
completion of tame bicategories exists and can be given an explicit description, which we now briefly
review. 
Let $\BimK$ be the bicategory with objects monads in $\ccatC$ and hom-category of maps between monads $(A,p)$ and $(B,q)$ given by the category of $(q,p)$-bimodules:
\[
\Hom\BimK{(A,p)}{(B,q)}
\ \defeq\ 
\LRMod\ccatC ApBq \,.
\]
For $(C,r)$ a third monad, the composition of bimodules is defined by 
\begin{align}
\label{bim-composition}
\begin{split}
\LRMod\ccatC BqCr
\times
\LRMod\ccatC ApBq
&\tto \LRMod\ccatC ApCr\\
(g,f) &\mto g\bimcomp q f
\end{split}
\end{align}
where $g\bimcomp q f$ is a $(r,p)$-bimodule whose underlying map is given by the reflexive coequaliser in $\hom\ccatC AC$:
\begin{equation}
\label{eq:bimcomp}
\begin{tikzcd}[column sep = large]
g \circ q \circ f \ar[r, bend right = 40, "g \circ \lambda"'] \ar[r, shift left = 0, bend left = 30, "\rho \circ f"] &
g \circ f \ar[l, "g \circ \eta \circ f"] \ar[r] &
g \bimcomp q f \,.
\end{tikzcd}
\end{equation}
This exists by the assumption that $\ccatC(A,C)$ has reflexive coequalisers, and can be equipped with the structure of an $(r,p)$-bimodule by the assumption that composition in $\ccatC$ preserves reflexive coequalisers.
The identity map of the object $(A,p)$ is given by $p \co A \to A$, viewed as a $(p,p)$-bimodule via the monad multiplication. 
It can be shown that $\BimK$ is tame and that there is a pseudofunctor
\[
J \co \ccatC \stto \BimK \,,
\]
defined by sending an object $A$ in $\ccatC$ to the identity monad $(A,\id_A)$ in $\BimK$, which is tame and fully faithful~\cite[Proposition~5.3.1]{GambinoN:opebaf}. 

\begin{prop}[{\cite[Theorem~5.4.2]{GambinoN:opebaf}}]
\label{prop:emk=bim} Let $\ccatC$ be a tame bicategory.
The tame bicategory $\Bim(\ccatC)$, equipped with the 
 pseudofunctor $J \co \ccatC\to \Bim(\ccatC)$, is an Eilenberg--Moore--Kleisli completion of $\ccatC$.
\end{prop}

We establish some immediate consequences of the
results in~\cite{GambinoN:opebaf} which will be useful later.

\begin{lem}
\label{lem:EMC-embedding}
Let $F \co \ccatC \to \ccatD$ be a pseudofunctor between tame bicategories. If
$F$ is fully faithful (thus tame), then $\EMK(F) \co \EMK(\ccatC) \to \EMK(\ccatD)$ is also fully faithful.
\end{lem}

\begin{proof}
This is a general fact about free completions, left to the reader.
The explicit description of $\EMK(\ccatC)$ in terms of $\Bim(\ccatC)$ makes it even easier to check.
\end{proof}

The next result does not seem to appear in the literature. It will be used in the proof of \cref{thm:opdrig=em-frig}.

\begin{prop}
\label{prop:EMC-embedding2}
Let $F \co \ccatC \to \ccatD$ be a pseudofunctor between tame categories.
Assume that $F$ is fully faithful and that $\ccatD$ is Eilenberg--Moore--Kleisli complete. Then $\EMK(\ccatC)$ is equivalent to the full sub-bicategory of $\ccatD$ spanned by the Eilenberg--Moore--Kleisli objects of monads in $\ccatC$.
\end{prop}

\begin{proof} 
We use \cref{prop:emk=bim} to prove the result with $\Bim(\ccatC)$ instead of $\EMK(\ccatC)$.
A tame bicategory $\ccatC$ is Eilenberg--Moore complete if and only if the inclusion $J \co \ccatC\to \Bim(\ccatC)$ is an equivalence~\cite[Proposition 5.3.9]{GambinoN:opebaf}. 
Under this equivalence, an object $(A, p)$ in $\Bim(\ccatC)$, given by an object $A$ in $\ccatC$ and a monad $p \co A \to A$ on it, is identified with its Eilenberg--Moore object $A^p$ in $\ccatC$.
Then the result follows from \cref{lem:EMC-embedding}.
\end{proof}

\section{The bicategory of 2-rigs} 
\label{sec:sym2rig}

\subsection{Cocontinuous monads}
\label{sec:monads-on-coc}

We denote by $\COC$ the large bicategory of categories with small colimits, cocontinuous functors, and natural transformations,  and by $\Pres$ its full subcategory spanned of (locally) presentable categories (we shall drop the `locally' from the name).

\begin{lem}
\label{lem:tame-coc}
The categories $\COC$ and $\Pres$ are tame.	
\end{lem}
\begin{proof}
The hom-categories have all small colimits and therefore reflexive coequalisers.
And the composition of cocontinuous functors preserves all colimits (precisely  because they are cocontinuous functors).
\end{proof}

\begin{lem}
\label{lem:monadicity-coc}
The forgetful pseudofunctor $\COC\to\CAT$ is pseudomonadic and creates Eilenberg--Moore objects.
\end{lem}

\begin{proof} As shown in~\cite{AnelM:sym2rcc}, the pseudofunctor is pseudomonadic. 
Therefore, it creates all 2-categorical limits~\cite{BlackwellR:twodmt,CreurerI:bectpm}.
In particular, it creates  Eilenberg--Moore objects since 
 Eilenberg--Moore objects are a form of lax limit~\cite{StreetR:fortm}.
\end{proof}

If $\cocC$ be a presentable category, a monad $P \co \cocC \to \cocC$ is said to be \emph{cocontinuous} if its underlying functor is cocontinuous.
Given a cocontinuous monad $P$ on a cocomplete category $\cocC$, we write
$\cocC^P$ for the category of Eilenberg--Moore algebras for $P$, which provides the Eilenberg--Moore object for $P$ in $\CAT$. 
Applying \cref{lem:monadicity-coc}, we get that this is also an Eilenberg--Moore object for $P$ in $\COC$. 
Moreover, when $\cocC$ is presentable, so is $\cocC^P$ \cite[2.78]{Adamek_Rosicky_1994}, thus leading
to the following proposition.

\begin{prop}
\label{prop:coc-emk-complete}
The bicategory $\Pres$ is Eilenberg--Moore complete and the forgetful functor $\Pres\to \CAT$ creates Eilenberg--Moore objects.
Moreover, being a tame bicategory, $\Pres$ is Eilenberg--Moore--Kleisli complete. \qed
\end{prop}

\begin{rmk} 
The existence of Eilenberg--Moore objects in $\Pres$ and the fact that they are calculated as in $\CAT$ can also be deduced from the results in
~\cite{BirdG:lim2cl} and~\cite{BirdG:flel2c}. First, \cite[Theorem~3.15]{BirdG:lim2cl} shows that $\Pres$ admits all weighted limits of retract type, which include inserters and equifiers, and that these are preserved by $U \co \Pres \to \CAT$.
Secondly, \cite[Proposition~1.1]{BirdG:flel2c} shows that inserters and equifiers allow us to construct Eilenberg--Moore objects.
\end{rmk}

Let $\catA$ be a small category and $P \co \pshA \to \pshA$ be a cocontinuous monad.
Let $\pshA^P$ be the category of $P$-algebras and 
\begin{equation}
\label{eq:em-adjunction-in-coc}
\begin{tikzcd}
\pshA \ar[r, shift left = 2, "F"] \ar[r, phantom, description, "\scriptstyle{\vdash}" rotate=90] & \pshA^P \ar[l, shift left = 2, "U"] 
\end{tikzcd}
\end{equation}
the associated adjunction in $\Pres$, so that the functor $U$ is monadic (and cocontinuous).
We define the category $\catA_P$ to be the image of the functor 
\[
\begin{tikzcd}
\catA \ar[r, "\embA"] & \pshA \ar[r, "F"] & \pshA^P \,. 
\end{tikzcd}
\] 
This is the full subcategory of $\pshA^P$ spanned by free algebras on the objects of $A$, thus
fitting into an image factorisation
\[
\begin{tikzcd}
\catA \ar[d, "\embA"'] \ar[r, two heads]  & \catA_P \ar[d, hook] \\
\pshA \ar[r, "F"'] & \pshA^P  \,.
\end{tikzcd}
\]
The category $\catA_P$ has the following explicit description.
Recall that a cocontinuous functor $P \co \pshA \to \pshA$ is equivalent to a ``coefficients matrix''  $\matrice P--  \co  \catA\op \times \catA \to \Set$. 
By construction, the objects of $\catA_P$ are those of $\catA$ and for $x$ and $y$ in $\catA$, we have the following description of the hom sets of $\catA_P$ in terms of the matrix $\matrice P--$:
\begin{align*}
\hom{\catA_P} a{a'}
&= \hom{\pshA^P}{P(a)}{P(a')}\\
&= \hom\pshA a{P(a')}\\
&= \hom\pshA a{\matrice P-{a'}}\\
&= \matrice Pa{a'} \,.
\end{align*}
The composition of $\catA_P$ is induced by the monad structure of $P$.

The left Kan extension of the inclusion $\catA_P \to \pshA^P$ along the inclusion $\catA_P\to \psh{\catA_P}$, defines a cocontinuous functor $K$ as follows
\begin{equation}
\label{eq:from-pshaA-to-pshAa}
\begin{tikzcd}
\catA_P \ar[r, hook] \ar[d, hook]
& \pshA^P
\\
\psh{\catA_P} \ar[ru, "K"']
\,.
\end{tikzcd}
\end{equation} 
It is a classical fact that $K$ is an equivalence, but we prove it for reference purposes in \cref{lem:monad-presheaves}.

\begin{lem}
\label{lem:surj-small=monadic}
If $f \co A\to B$ is an essentially surjective functor between small categories, then the adjunction $f_! \co \pshA \rightleftarrows \pshB \co f^*$ is monadic in $\COC$.
In particular, the functor $P=f^*f_!$ is a cocontinuous monad and $\pshB=\pshA^P$.
\end{lem}
\begin{proof}
The functor $f^*$ is conservative since $f$ is surjective.
Since $f^*$ has a left adjoint and is cocontinuous, the result follows from the monadicity theorem.
\end{proof}

\begin{lem}
\label{lem:monad-presheaves}
Let $\catA$ be a small category and $P \co \pshA \to \pshA$ a cocontinuous monad.
The functor $K \co \psh{\catA_P} \to \pshA^P$ of~\eqref{eq:from-pshaA-to-pshAa} is an equivalence.
\end{lem}
\begin{proof} 
The canonical functor $i \co \catA\to \catA_P$ induces an adjunction in $\Pres$
\[
\begin{tikzcd}
\pshA
\ar[r, shift left = 2, "i_!"]
\ar[r, phantom, description, "\scriptstyle{\vdash}" rotate=90]
\ar[from=r, shift left = 2, "i^*"] 
& \psh{\catA_P}
\,.
\end{tikzcd}
\]
Since $i \co \catA\to \catA_P$ is surjective by definition of $\catA_P$, this adjunction is monadic by \cref{lem:surj-small=monadic}.
This provides a canonical equivalence $K'\co\psh{\catA_P} \to \pshA^P$ which commutes 
with the free-algebra functor $F$ and $i_!$.
Therefore, $K'$ must be the left Kan extension of the canonical inclusion $\catA_P\to \pshA^P$, that is $K'=K$.
\end{proof}

\begin{rmk}
\label{rmk:naturality-cc-mnd}
Define a morphism of cocontinuous monads $(A,P)\to (B,Q)$ as a pair $(f,\phi)$ where $f \co A\to B$ is a functor and $\phi$ a morphism of monads $\phi \co P\to f^*Qf_!$ in $\pshA$
(or equivalently a natural transformation $\phi' \co f_!\circ P\to Q\circ f_!$ satisfying suitable conditions).
With the obvious notion of 2-cell, this defines a bicategory $\mathcal{M}$ of cocontinuous monads and one can show that the equivalence of \cref{lem:monad-presheaves} is functorial on $\mathcal{M}$.
\end{rmk}

For a small category $\catA$, we shall say that the cocomplete category $\Psh\catA$ is \emph{free on $\catA$}.
We denote by $\FsetCOC\subseteq\FCOC\subseteq\COC$ the full sub-bicategories of $\COC$ spanned by cocomplete categories that are free on a small set and free on a small category, respectively.
Since $\COC$ is tame, these are also tame bicategories.

\begin{prop}
\label{prop:free-coc}
The Eilenberg--Moore--Kleisli completion of $\FsetCOC$ is $\FCOC$.
\end{prop}
\begin{proof}
By \cref{lem:monad-presheaves}, $\FCOC$ is closed under Eilenberg--Moore--Kleisli objects and this proves the inclusion $\FsetCOC\subseteq\EMK(\FCOC)$.
Conversely, if $C$ is a small category with set of objects $\catA$, then the matrix $\catA\times \catA\to \Set$ of hom sets of $\catC$ defines a continuous monad in $\pshA$.
By \cref{lem:monad-presheaves}, the Eilenberg--Moore--Kleisli object of this monad is $\psh{\catA_P} \cong \pshC$ and this shows that $\FCOC\subseteq\EMK(\FsetCOC)$.
\end{proof}

\subsection{Symmetric 2-rigs} 
\label{sec:2-rigs}

Before we define 2-rigs, we need to recall some basic material on symmetric monoidal categories.
Let $\FinBij$ be the groupoid of finite sets and bijections.
We put $\set 0 = \emptyset$ and $\set n = \{1,\dots,n\}$, for $n\geq1$.
We write~$\mathfrak S_n\subseteq\FinBij$ the full subcategory spanned by the object $\underline n$, \ie the
$n$-th symmetric group, and~$\mathfrak S \subseteq\FinBij$ the full subcategory spanned by all the $\set n$, for $n\geq0$.
The inclusion $\mathfrak S\subseteq \FinBij$ is an equivalence of categories and $\mathfrak S$ is a skeleton for $\FinBij$.

We denote by $\SMCat$ the bicategory of small symmetric monoidal categories, symmetric strong monoidal functors, and monoidal transformations.
This bicategory comes with a biadjunction 
\begin{equation}
\label{eq:adj:cat-smcat}
\begin{tikzcd}
\Freesmc  \co  \Cat 
\ar[r, shift left = 2]
\ar[r, phantom, description, "\scriptstyle{\vdash}" rotate=90]
\ar[from=r, shift left = 2]
&\SMCat  \co  \sfU
\end{tikzcd}
\end{equation}
where $\sfU$  is the forgetful pseudofunctor and its left biadjoint $\Freesmc$ is constructed by the bicolimit formula in $\Cat$
\[
\freesmcC
\ \defeq\ 
\bicolim_{n\in \mathfrak S}C^n
\ \simeq\ 
\bicolim_{N\in \FinBij}C^N
\,,
\]
equipped with the obvious product.
Explicitly, an object of $\freesmcC$ is a family $\symelt cn$ of objects in $C$.
A morphism $\symelt cn\to \symelt {c'}m$
is a pair $(f,\phi)$ where $f \co \set n\to \set m$ is a bijection 
and $\phi_i  \co  c_i\to c'_{f(i)}$ is a family of arrows in $C$ for $i=1,\dots,n$.
In particular, we have $\freesmc 1 = \mathfrak S \simeq \FinBij$.
The biadjunction is pseudomonadic~\cite{BlackwellR:twodmt}.

\medskip
We now turn to the definition and study of 2-rigs.
\begin{defn} \leavevmode
\begin{itemize}
\item A \emph{(symmetric) 2-rig} is a presentable category equipped with a symmetric monoidal structure such that the tensor product is a cocontinuous functor in each variable.
\item A \emph{2-rig morphism} is a symmetric strong monoidal cocontinuous functor.
\item A \emph{2-rig transformation} is a symmetric monoidal transformation. 
\end{itemize}
\end{defn}

All 2-rigs considered in this paper will be symmetric, \ie~their underlying monoidal structure is symmetric.
Thus, we simply speak of 2-rigs rather than of symmetric 2-rigs. 
We define $\Rig$ to be the bicategory of 2-rigs, 2-rig morphisms and 2-rig transformations.

\begin{ex}\leavevmode
\label{ex:rig}
\begin{enumerate}[label=(\alph*)]

\item\label{ex:rig:convolution} If $\moncatM$ is a small symmetric monoidal category, then $\pshM$ is a 2-rig for the Day convolution product. Such 2-rigs will be called \emph{convolution 2-rigs}. We denote by $\CRig\subseteq\Rig$ the full sub-bicategory of convolution 2-rigs. By construction, it is the image of a pseudo-functor $\Psh \co \SMCat\to \Rig$.

\item\label{ex:rig:free} In particular, when $\moncatM=\freesmcA$ is a free symmetric monoidal category on a small category $\catA$, every category $\freerigA$ is a 2-rig. Such 2-rigs will be called \emph{free}. 
The construction of the free 2-rig $\freerigA$ defines a partial left pseudoadjoint to the forgetful pseudofunctor $\Rig \to \CAT$ along the inclusion $\Cat\subseteq \CAT$.
We denote by $\FRig\subseteq\Rig$ the full sub-bicategory of free 2-rigs, and by $\FsetRig\subseteq\FgpdRig\subseteq\FRig$ the full sub-bicategories of free 2-rigs generated by a set and by a groupoid.

\end{enumerate}	
\end{ex}
By definition, the categories $\FRig$ and $\CRig$ fit into image factorisations as follows:
\begin{equation}
\label{eq:cat-smcat-2rig}
\begin{tikzcd}
\Set
	\ar[d,two heads, "\Psh \circ \Freesmc"']
	\ar[r, hook]
&\Cat
	\ar[d,two heads, "\Psh \circ \Freesmc"']
	\ar[r, "\Freesmc", hook]
&\SMCat
	\ar[d,two heads, "\Psh"]
\\
\FsetRig 
	\ar[r,hook]
&\FRig 
	\ar[r,hook]
&\CRig
	\ar[r,hook]
&\Rig
\,. & 
\end{tikzcd}
\end{equation}

We now recall a more explicit description of the categories $\FRig$ and $\FsetRig$ from \cite{GambinoN:opebaf}.
For two small categories $\catA$ and $\catB$, the category of \emph{symmetric sequences} from $\catA$ to $\catB$ is defined to be
\[
\CatSym(A,B) \defeq \fun{\freesmcB\op\times\catA}\Set\,.
\]
We then have a  chain of equivalences
\begin{align*}
\fun{\freesmcB\op\times\catA}\Set 
&\simeq \Fun{\catA}{\fun{\freesmcB\op}\Set}\\
&\simeq \Fun{\catA}\freerigB\\
&\simeq \hom\Rig\freerigA\freerigB\,.
\end{align*}
Two symmetric sequences $F \co \freesmcB\op\times\catA\to\Set$ and 
$G \co \freesmcC\op\times\catB\to \Set$ can be composed.
By first extending $G$ to a functor $G^\times \co \freesmcC\op\times\freesmcB\to\Set$ which is symmetric strong monoidal in the second variable, and then taking the coend of $G^\times$ and $F$ over $\freesmcB$ 
\begin{equation}
\label{eq:composition-sym-seq}
G\bimcomp {}  F
\ \defeq\ 
G^\times\coend{\freesmcB} F
\ =\ 
\int^{\vec y \in \freesmcB} \Smatrice Fy- \times G^{\times \vec y} \,,
\end{equation}
where $G^{\times \vec y}= G(y_1) \times \ldots \times G(y_m)$, for ${\vec y} = \angles{y_1, \ldots, y_m}$.
This composition defines a functor 
\[
(-) \bimcomp {}  (-)  \co  \fun{\freesmcB\op\times\catA}\Set
\times 
\fun{\freesmcC\op\times\catB}\Set
\stto 
\fun{\freesmcC\op\times\catA}\Set
\]
and a straightforward computation shows that it corresponds to the composition of 2-rig morphisms under the equivalences $\fun{\freesmcB\op\times\catA}\Set \simeq \hom\Rig\freerigA\freerigB$.
In this way, it is possible to show that categories and symmetric sequences form a bicategory $\CatSym$ which is equivalent to $\FRig$~\cite[Theorem~2.4.4 and Theorem~3.2.2]{GambinoN:opebaf}.\footnote{In~\cite[page~45]{GambinoN:opebaf}, $\CatSym$ is defined as the opposite of
 the bicategory $S\text{-}\mathsf{Dist}$.}
We record this result for future reference.

\begin{prop}
\label{prop:freerig-biequivalence}
The following bicategories are biequivalent:
\begin{enumerate}
\item the bicategory $\FRig$ of free 2-rigs,
\item the bicategory $\CatSym$ of categorical symmetric sequences.
\end{enumerate}
Moreover, this equivalence restricts to an equivalence $\FsetRig\simeq\SetSym$ where $\SetSym\subseteq\CatSym$ is the full sub-bicategory spanned by sets.
\end{prop}

Recall from~\cite{BirdG:lim2cl} that the 2-category  $\Pres$ admits a symmetric monoidal closed structure, whose tensor product classifies cocontinuous functors in two
variables. We write $C \otimes D$ for this tensor product, whose unit is the category $\Set$. 
A 2-rig $\rigR$ can be defined equivalently as a symmetric pseudomonoid in $(\Pres,\otimes,\Set)$ and the forgetful pseudofunctor
$U \co \Rig \to \Pres$ 
sending a 2-rig to its underlying small cocomplete category can be regarded as an instance of the forgetful pseudofunctor mapping a pseudomonoid to its underlying object.
As discussed in~\cite{AnelM:sym2rcc}, this forgetful pseudofunctor has a left biadjoint, which we write $\Sym \co \Pres \to \Rig$ and refer as the \myemph{symmetric algebra} pseudofunctor, and we have the following result. 

\begin{prop} 
\label{thm:sym-alg-coc}
There exists a pseudomonadic biadjunction
\[
\begin{tikzcd}
\Pres \ar[r, shift left = 2, "\Sym"] \ar[r, phantom, description, "\scriptstyle{\vdash}" rotate=90] 
& \Rig  \,. \ar[l, shift left = 2, "\mathsf{U}"] 
\end{tikzcd}
\]
\end{prop}

\begin{rmk}
\label{rmk:adj-smcat-rig}
The bicategory $\Rig$ comes with another forgetful functor to the bicategory $\SMCAT$ of large symmetric monoidal categories.
This functor has a relative left biadjoint, defined on the full subcategory $\SMCat\subseteq\SMCAT$ of small symmetric monoidal categories, sending such a category $\moncatM$ to the category of presheaves $\pshM=\fun{M\op}\Set$ equipped with the Day convolution product.
\begin{equation}
\label{eq:adj-smcat-rig}
\begin{tikzcd}
&\Rig\ar[d,"\sfU"]\\
\SMCat \ar[r,hook]\ar[ru,"\Psh"] & \SMCAT
\end{tikzcd}
\end{equation}
See~\cite{AnelM:sym2rcc} for additional details.
\end{rmk}

\subsection{Tameness}

The aim of this section is to show that the bicategory $\RIG$ is tame, in the sense of \cref{defn:tame}.
In fact, we shall in fact prove a slightly stronger result.

Recall that a category $I$ is \emph{sifted} if the diagonal functor $\delta_2 \co I\to I\times I$ is a cofinal functor. When this is then case, then all the diagonal functors $\delta_n \co I\to I^n$, for $n > 2$, are cofinal functors. The terminal category $\mathsf{1}$ is sifted. The class of sifted colimits is the class of colimits indexed by sifted categories. The next lemma is crucial.

\begin{lem}
\label{lem:forget-sifted}
Let $\rigR$ and $\rigS$ be 2-rigs. The forgetful functor 
\[
U_{\rigR, \rigS} \co \hom\Rig\rigR\rigS \stto \fun\rigR\rigS
\] 
creates sifted colimits.
\end{lem}
\begin{proof}
The colimits indexed by $I=1$ are created by $U$ if and only if $U$ is conservative.
We leave to the reader the proof that $U$ is indeed conservative.
Recall that a conservative functor creates a colimit if and only if it lifts it.
So, it is sufficient to show that $U$ lifts sifted colimits.
The category $\fun\rigR\rigS$ is cocomplete since $\rigS$ is, in particular it has sifted colimits.
Let $I$ be a sifted category, $f \co I\to \hom\Rig\rigR\rigS$ be a diagram of rig morphisms, 
and $f = \colim f_i$ be its colimit calculated in $\fun\rigR\rigS$.
We need to prove that the colimit cocone $(f_i \to f)_{i \in I}$ can be enhanced into a colimit in $\hom\Rig\rigR\rigS$.
The functor $f$ can be equipped with a 2-rig morphism structure by means of the following canonical isomorphisms (for any $n\geq 0$)
\begin{align*}
f(x_1\otimes \dots \otimes x_n)
& = \colim_i f_i(x_1\otimes \dots \otimes x_n)\\
&\cong \colim_i f_i(x_1) \otimes \dots \otimes f_i(x_n) && \text{since $f_i$ is a 2-rig morphism}\\
&\cong \colim_{i_1,\dots,i_n} f_{i_1}(x_1) \otimes \dots \otimes f_{i_n}(x_n)&& \text{since $\delta_n \co I\to I^n$ is cofinal}\\
&\cong \Big(\colim_{i_1} f_{i_1}(x_1)\Big) \otimes \dots \otimes \Big(\colim_{i_n} f_{i_n}(x_n)\Big)&& \text{since $\otimes$ preserves colimits}\\
&= f(x_1) \otimes \dots \otimes f(x_n)\,.
\end{align*}
We leave the reader to check that these satisfy the coherence conditions for a strong monoidal
functor, thereby making it into a 2-rig morphism. This concludes the proof that $U$ lifts sifted colimits.
\end{proof}

\begin{prop}
\label{prop:rigs-local-sifted}
The $2$-category $\Rig$ has local sifted colimits, 
\ie~for every $\rigR$ and $\rigS$ in $\Rig$, the hom-category $\hom\Rig\rigR\rigS$ has sifted colimits and the composition functors of $\Rig$ preserve sifted colimits in each variable.
\end{prop}

\begin{proof}
\Cref{lem:forget-sifted} shows that $\Rig$ has local sifted colimits.
For preservation of colimits, let $\rigR$, $\rigS$ and $\rigT$ be 2-rigs 
and consider the composition functor 
\[
\circ \co \hom\Rig\rigS\rigT  \times \hom\Rig\rigR\rigS \stto \hom\Rig\rigR\rigT \,.
\] 
For a 2-rig morphism $g \co \rigS \to \rigT$, the functor $g \circ (-)$ defined
by composition with $g$ preserves sifted colimits since colimits in functor categories
are calculated pointwise and $g$ is cocontinuous. For a 2-rig morphism
$f \co \rigR \to \rigS$, the functor $(-) \circ f$ defined by composition with $f$
preserves sifted colimits since colimits are calculated pointwise.
\end{proof}

\begin{cor}
\label{cor:RIG-is-tame}
The $2$-category $\Rig$ is tame.
\end{cor}

\begin{proof} The claim follows from \cref{prop:rigs-local-sifted} since reflexive coequalizers are sifted colimits.
\end{proof}

\begin{rmk}
Recall that any full sub-bicategory of a tame category is tame (\cref{rmk:sub-tame=tame}).
Under the equivalences $\FRig\simeq\CatSym$ and $\FsetRig \simeq \SetSym$ (\cref{prop:freerig-biequivalence}) we recover that $\CatSym$ and $\SetSym$ are tame bicategories \cite[Corollary 4.4.9]{GambinoN:opebaf}.
\end{rmk}

\subsection{Monads on 2-rigs}

\begin{defn} \label{thm:2rigmonad}
Let $\rigR$ be a 2-rig. A \emph{$2$-rig monad} on $\rigR$ is a 
monad $(P,\mu,\eta) \co \rigR \to \rigR$ such that $P \co \rigR \to \rigR$ is a 2-rig morphism, and $\mu \co P \circ P \Rightarrow P$ and $\eta \co \id_\rigR \Rightarrow P$ are 2-rig transformations.
\end{defn}

For a 2-rig $\rigR$ and a monad on $\rigR$ as above, we write $\rigR^P$ for the category of Eilenberg--Moore algebras for $P$ and $\rigR_P$ for the Kleisli category of $P$, for the moment considered as mere categories.
We establish the last main result of this section.

\begin{thm}
\label{thm:rig-em-complete} 
The $2$-category $\Rig$ is Eilenberg--Moore--Kleisli complete, and Eilenberg--Moore--Kleisli objects can be computed as Eilenberg--Moore objects in $\CAT$.
\end{thm}

\begin{proof} 
We follow the same argument as in \cref{prop:coc-emk-complete}.
By \cref{thm:sym-alg-coc}, $\Rig$ is pseudomonadic over $\Pres$ and therefore the forgetful functor $U \co \Rig \to \Pres$ creates all 2-categorical limits, in particular Eilenberg--Moore objects.
Then the creation of Eilenberg--Moore objects by the forgetful pseudofunctor to $\CAT$ follows from \cref{prop:coc-emk-complete}.
Finally, these are Eilenberg--Moore--Kleisli objects because $\Rig$ is tame by \cref{cor:RIG-is-tame}.
\end{proof}

The following lemma is left to the reader.

\begin{lem}
\label{lem:image-monoidal}
For a symmetric monoidal functor $f \co \moncatM\to \moncatN$, the image $\mathsf{Im}(F)\subseteq\moncatN$ of $f$ computed in~$\Cat$ is closed under the monoidal structure.
\end{lem}

Recall the definition of the functor $K$ in~\eqref{eq:from-pshaA-to-pshAa}.
The next lemma extends \cref{lem:monad-presheaves}.

\begin{lem}  \label{lem:monad-presheaves-monoidal}
Let $\moncatM$ be a symmetric monoidal category and 
 $P \co \pshM \to \pshM$ be a 2-rig monad. The equivalence
$K \co \psh{\moncatM_P} \to \pshM^P$ is a symmetric strong monoidal functor.
\end{lem}

\begin{proof}
The fact that $K$ is an equivalence is \cref{lem:monad-presheaves}.
We need to see that it is a symmetric strong monoidal functor.
By construction, $\moncatM_P$ is the image of the monoidal functor $\moncatM\to \pshM\xrightarrow P \pshM$.
By \cref{lem:image-monoidal}, it is a symmetric monoidal full subcategory of $\pshM$.
Thus, the functor $K$ is the left Kan extension of the inclusion $i \co \moncatM_P \hookrightarrow \pshM^P$, which is strong monoidal, and hence $K$ is again a strong monoidal functor \cite{DayB:clocf,ImG:unipcm}.
\end{proof}

\begin{prop}
\label{prop:drig-em-complete} 
The sub-bicategory $\CRig\subseteq\Rig$ is closed under Eilenberg--Moore--Kleisli objects.
\end{prop}

\begin{proof} 
Let $\moncatM$ be a symmetric monoidal category and $\pshM$ the associated convolution 2-rig.
Let $P \co \pshM\to \pshM$ be a monad in $\RIG$.
By \cref{lem:monad-presheaves-monoidal}, the Eilenberg--Moore category $\pshM^P$
of $m$ is equivalent to $\psh{\moncatM_P}$ where $\moncatM_P$ is the category of free $P$-algebras on the objects of $\moncatM$. 
This shows that the Eilenberg--Moore object of a 2-rig monad on a convolution 2-rig is again a convolution 2-rigs.
\end{proof}

\section{Operadic 2-rigs}
\label{sec:operadic-2-rigs}

The aim of this section is to expand the diagram in~\eqref{eq:cat-smcat-2rig} into a diagram
\begin{equation}
\label{eq:cat-opd-smcat-2rig}
\begin{tikzcd}[row sep = large]
\Cat
\ar[d,two heads]
\ar[r, hook]
&\Opd
\ar[d,two heads]
\ar[r,"\Env"]
&\SMCat
\ar[d,two heads]
\ar[dr,"\Psh", bend left]
\\
\FRig \ar[r,hook]
& \OpdRig \ar[r,hook]
& \CRig  \ar[r, hook]
& \Rig \,, 
\end{tikzcd}
\end{equation}
where $\Opd$ is the bicategory of  operads and $\OpdRig$ a bicategory of rigs generated by operads, both to be introduced below.

\subsection{Operads}
\label{sec:coloured-operads}

\begin{defn} \label{defn:coloured-operad}
Let $A$ be a set. A \emph{coloured operad} $\opdP$ with set of colours $A$ consists of:
\begin{itemize}
\item a functor
\begin{align*} 
\opdP  \co  \freesmcA\op \times  A  & \tto \Set \\
 ( \symelt an, a)  & \mto \matrice\opdP{a_1, \ldots, a_n}a \,,
\end{align*} 
called the \emph{functor of operations} of $\opdP$,

\item a natural transformation with components 
\[
\matrice\opdP{\vec a''_1}{a'_1}
\times\ldots\times
\matrice\opdP{\vec a''_n}{a'_n}
\times
\matrice\opdP{\vec a'}x
\stto
\matrice\opdP{\vec a''_1 \otimes \ldots \otimes \vec a''_n}a
\,,
\]
for $a$ in $\catA$, 
$\vec a' = \symelt{a'}n$ in $\freesmcA$, 
$\vec a''_1, \ldots, \vec a''_n$ in $\freesmcA$, 
and $n$ in $\mathbf{N}$, called the \emph{composition} of $\opdP$,
\item elements $\id_a$ in $\matrice\opdP aa$, for $a$ in $\catA$, called the \emph{identity operations} of $\opdP$.
\end{itemize}
These data are subject to associativity and unitality axioms~\cite{yau2016colored}.
\end{defn}

As mentioned, we shall often omit  `coloured' and simply speak of operads.
For an operad $\opdP$ with set of colours $A$ as in \cref{defn:coloured-operad}, an element $f$ in $\matrice\opdP{a_1, \ldots, a_n}a$ will be called an \emph{operation} of $\opdP$ with \emph{input arities} $a_1, \ldots, a_n$ and \emph{output arity} $a$ and written~$f \co a_1, \ldots, a_n \to a$. For $n = 1$, we speak of a \emph{unary operation}.  Note that, even if $A$ is a set, $\freesmcA$ is a category (a groupoid, actually) and hence, by the functoriality of $F$, we have actions of the symmetric
groups on the sets of operations. Indeed, for a permutation $\sigma$ in $\mathfrak{S}_n$, we get a function
\[
\matrice\opdP{a_1, \ldots, a_n}a \stto \matrice\opdP{a_{\sigma(1)}, \ldots, a_{\sigma(n)}}a  \,.
\]
When an operad $\opdP$ has a single colour, we denote the set of $n$-ary operations simply by $\opdP(n)$ ($n\neq0$).
We shall say that an operad is \emph{discrete} if it has only identity operations.

\begin{defn} Let $\opdP$ be a operad with set of colours $\catA$, $\opdQ$ be a operad with set of colours $B$. A \emph{coloured operad morphism} $(f, \phi) \co \opdP \to \opdQ$ consists of:
\begin{itemize}
\item a function $f \co \catA \to \catB$, 
\item a natural transformation  $\phi \co \opdP \Rightarrow \opdQ \circ (\freesmc F \times F)$, with components
\begin{align*} 
\phi_{a_1, \ldots, a_n; x} \co  \matrice\opdP{a_1,\ldots, a_m}a & \stto \matrice\opdQ{f(a_1), \ldots, f(a_m)}{f(a)} 
\end{align*}
for $a_1, \ldots, a_n$, $a$ in $\catA$.
\end{itemize}
These data are subject to a functoriality condition.
\end{defn}

A morphism of operads $(f, \phi) \co \opdP\to \opdQ$ is \emph{fully faithful} if the natural transformation $\phi$ is invertible.
From now on, we shall refer to an operad morphism as above simply by the name
of its underlying function on objects when this does not cause confusion.
Finally, we introduce 2-cells. 

\begin{defn} \label{defn:opd-2cell} 
Let $f, g \co \opdP \to \opdQ$ be morphisms of operads.
 A  \emph{transformation}
$\alpha \co f \Rightarrow g$ consists of unary operations $\alpha_a \co f(a) \to g(a)$ in $\opdQ$ satisfying a naturality condition. 
\end{defn}

The full details of these definitions can be found in 
\cite{yau2016colored,ElmendorfA:rinmai}.
We write $\Opd$ for the bicategory of  operads, morphisms, and transformations.

\begin{rmk}
\label{rmk:1-cat-opd}
Operads and their morphisms also form a 1-category, that we shall denote by $\Opdone$.
The set of colors of an operad defines a functor $\Opdone\to \Set$.
\end{rmk}

\begin{ex}\leavevmode
\label{ex:opd}
\begin{enumerate}[label=(\alph*)]
\item\label{ex:opd:cat} \label{def:opd-surj} 
Any small category $C$ can be seen as an operad $\opdO_C$ with only unary operations.
When $C=1$ is the terminal category, we shall call $\opdO_1$ the \emph{operad of objects}.
Conversely, the unary operations  of an operad $\opdP$ defines a category $\opdP_1$, called the \emph{underlying category of $\opdP$}.
This provides a pseudoadjunction
\begin{equation}
\label{eq:adj-cat-opd}
\begin{tikzcd}
\Cat
	\ar[r, shift left = 2, "i"] 
	\ar[r, phantom, description, "\scriptstyle{\vdash}" rotate=90] & 
\Opd  \,.
	\ar[l, shift left = 2, "(-)_1"] 
\end{tikzcd}
\end{equation}
We shall say that a morphism of operads $f \co \opdP\to \opdQ$ is \emph{essentially surjective} if the induced functor $f_1 \co \opdP_1\to \opdQ_1$ between underlying categories is essentially surjective.

\item\label{ex:opd:com}
The \emph{(unital) commutative operad} $\Com$ is the single-coloured operad defined by $\Com(n)=1$ (terminal object in $\Set$), with the obvious composition. This is the terminal (symmetric) operad, so every other operad $\opdP$ has a canonical (and unique) morphism $\opdP\to \Com$.

\item The \emph{non-unital commutative operad} $\Comnu$ has the same definition as $\As$ but with $\Comnu(0)=\emptyset$.

\item The \emph{unital operad} $\Ezero$ is the single-coloured operad defined by $\Ezero(0)=1$ and $\Ezero(n)= \empty$ if $n\geq1$, with the obvious composition.

\item The \emph{(unital) associative operad} $\As$ is the single-coloured operad defined by $\As(n)=\mathfrak S_n$, the symmetric group on $n$ elements, with its canonical action by translation.
The composition is induced by the canonical maps $\mathfrak S_p\times \mathfrak S_q\to \mathfrak S_{p+q}$.
There is a canonical morphism of operads $\Ezero\to \As$.

\item The \emph{non-unital associative operad} $\Asnu$ has the same definition as $\As$ but with $\As(0)=\emptyset$.
There is a canonical morphism of operads $\Asnu\to \As$.

\item The \emph{associative ``left module'' operad} $\AsMod$ is the operad with two colours $A=\{a,m\}$, and with operations 
\[
\matrice\AsMod {a, \dots,a}a = \matrice\AsMod {a,\dots,a,m}m = \As(n)\,.
\]
when there is a single occurence of $m$ and $n-1$ occurrences of $a$ in the contravariant variables in $\matrice\AsMod {a,\dots,a,m}m$,
and empty otherwise.
The composition is induced by that of $\As$, \cf \cite[Definition~4.2.1.1]{LurieJ:higa}.
More generally, the construction can be done for any operad $\opdP$ instead of $\As$\cite[12.3.1]{LodayJL:algo}.

\item  For any two operads $\opdP$ and $\opdQ$ there exists an operad $\fun\opdP\opdQ$ whose algebras are the $(\opdP,\opdQ)$-bimodules \cite{GambinoN:opebaf}. Even when $\opdP$ and $\opdQ$ have a single colour, the operad $\fun\opdP\opdQ$ has several colours.

\item  For a given set of colours $A$, there is also a coloured operads whose algebras are the operads coloured by $A$, see \cite[Example 2.2.23]{LeinsterT:higohc} and~\cite[Example~1.5.6]{BergerC:rescorha}.
\end{enumerate}	
\end{ex}

\begin{rmk}
\label{rmk:set-of-colours}
Taking the set of objects of an operad does not determine a pseudofunctor $\Opd\to \Set$ since an equivalence of operads need not induce a bijection between the sets of colours.
The situation is the same as with the set of objects of a category, which does not defined a pseudofunctor $\Cat\to \Set$.
Because of this, all the constructions using explicitly the set of objects are not immediately pseudofunctorial on $\Opd$ (\eg  \cref{prop:em-for-freerig} below).
Such constructions are however functorial on the 1-category $\Opdone$ of \cref{rmk:1-cat-opd}.

To go around this limitation, one can replace the set of colours by a more intrinsic object: either the category $\opdP_1$ of all unary operations, or  the groupoid $\opdP_1^\mathrm{inv}$ of invertible unary operations.
Alternatively, one can consider the more general notion of a substitude, which is defined like an operad but where the set $A$ of colours is allowed to be a category. Note, however, that this notion comes with a different notion of 2-cell~\cite{DayB:abssec,DayB:laxmpo}.
The bicategory $\Opd$ is a coreflective sub-bicategory bicategory of substitudes.
The coreflection is given by replacing the category of colours by the category of unary operations
(see \cite[5.17]{BataninM:regpsf}).

Despite the aforementioned limitation, we have preferred to use the sets of objects, which is the classical, and more intuitive, definition of an operad.
Some of our results (like \cref{prop:em-for-freerig}) 
could be strengthened, but we felt that the gain in generality would not offset the increase in complexity.
\end{rmk}

\begin{rmk}
\label{rmk:opd=monad}
An equivalent definition of an operad $\opdP$ with set of colours $A$ is as a  monad 
\[
{\monad\opdP \co \freerigA\to \freerigA}
\]
 in the bicategory $\SetSym = \FsetRig$ 
of \cref{prop:freerig-biequivalence}.
More precisely, the symmetric sequence of $\monad\opdP$ is given by the functor of operations of the operad $\opdP$ , the composition of the operad corresponds to the multiplication of the monad, the identities of the operad to the unit of the monad, see~\cite{BaezJ:higda} and~\cite[Example~4.1.4]{GambinoN:opebaf} for details. 
However, despite this coincidence, the category $\Opd$ is quite different from the category $\mathsf{Mnd}(\SetSym)$ of monads in $\SetSym = \FsetRig$ as defined in \cite{StreetR:fortm,LackS:fortmII}.
\end{rmk}

\medskip
The Hermida biadjunction between non-symmetric  operads
 (\ie~multicategories) 
and monoidal categories~\cite{HermidaC:repm} extends to a biadjunction between  operads 
(\ie~symmetric multicategories)
and symmetric monoidal categories, as shown in~\cite{ElmendorfA:percma} and \cite[\S5.2]{BataninM:regpsf}:
\begin{equation}
\label{eq:hermida-opd-smoncat}
\begin{tikzcd}
\Opd
	\ar[r, shift left = 2, "\Env"] 
	\ar[r, phantom, description, "\scriptstyle{\vdash}" rotate=90] & 
\SMCat  \,.
	\ar[l, shift left = 2, "\END"] 
\end{tikzcd}
\end{equation}

We provide some explicit definitions, as they will be useful in the rest of the paper.
The right biadjoint sends a symmetric monoidal category $\moncatM$ to the \emph{endomorphism operad}, written $\EndM$ and defined as follows.
The set of colours of $\EndM$ is the set of objects of $\moncatM$, while its operations are defined by letting
\[
\matrice\EndM{x_1, \ldots, x_n}x
\ \defeq\ 
\hom\moncatM{x_1 \otimes \cdots  \otimes x_n}x \,.
\]
The composition is defined using the functoriality of the tensor product of $\moncatM$.
The identities operations of $\EndM$ are the identities arrows of $\moncatM$.

The left biadjoint sends an operad $\opdP$ to its \emph{enveloping symmetric monoidal category}, written $\envP$ and defined as follows. 
The objects of $\envP$ are sequences of objects of $\catA$.
Given two sequences $\symelt am$ and $\symelt {a'}n$, 
a map $(f,p) \co \symelt am \to \symelt{a'}n$ consists of
a function $f \co \set m \to \set n$
and operations 
$p_j$ in $\matrice\opdP{\angles{a_i}_{i \in \phi^{-1}(j)}}{a'_j}$, for every $j$ in $\set n$. 
Here, the sequence $\angles{a_i}_{i \in \phi^{-1}(j)}$ is regarded as a subsequence of $\symelt am$, \ie~we consider the ordering of the objects as they appear in the original sequence. 
In other terms
\begin{equation}
\label{eq:env}
\Hom{\envP}{\symelt am}{\symelt{a'}n}
\ =\ 
\coprod_{\phi:\underline m\to \underline n}
\prod_{j=1}^n\matrice\opdP{\angles{a_i}_{i \in \phi^{-1}(j)}}{a'_j}
\end{equation}

For an operad $\opdP$, the unit of the biadjunction is the morphism of operads $\eta_\opdP \co \opdP \to \End \envP$ which sends an object $a$ in $\opdP$ to the one-element sequence~$\angles{a}$ in $\End \envP$. 
The biadjointness means that composition with 
$\eta_\opdP$ induces an equivalence of categories
\[
\begin{tikzcd}[column sep = 2cm]
\hom\SMCat\envP\moncatM \ar[r, "(-) \circ \eta_\opdP"]  & \hom\Opd\opdP\EndM 
\,,
\end{tikzcd}
\]
for every symmetric monoidal category $\moncatM$. 
Explicitly, this means that for every morphism of operads $f \co \opdP \to \EndM$ we
have an essentially unique symmetric strong monoidal functor $f^\dagger \co \envP \to \moncatM$
making the following diagram commute up to unique  natural isomorphism:
\[
\begin{tikzcd}
\opdP \ar[r, "\eta_\opdP"] \ar[dr, "f"']  & \End \envP \ar[d, "\End{f^\dagger}" ] \\
 & \EndM \,.
\end{tikzcd}
\]
The adjoint transpose $f^\dagger \co \envP \to \moncatM$ is defined on objects by mapping
a sequence $\symelt am$ in $\envP$ to the tensor product $f(a_1) \otimes \cdots \otimes f(a_m)$ in $\moncatM$.

\begin{ex}\leavevmode
\label{ex:env}
\begin{enumerate}[label=(\alph*)]
\item\label{ex:env:cat} When $\opdP=C$ is a category, \ie has only unary operations, we have $\envC = \freesmcC$.
This can be seen by the equation in~\eqref{eq:env},
or by noticing that the composition of the adjunctions in~\eqref{eq:adj-cat-opd}
and~\eqref{eq:hermida-opd-smoncat}
gives the adjunction $\Freesmc  \co  \Cat \rightleftarrows \SMCat \co \sfU$ of \eqref{eq:adj:cat-smcat}.
In particular, when $C=1$ is the terminal category,  the envelope of the operad of objects is $\env{\opdO_1}=\freesmc1$ which is equivalent to the category $\FinBij$ of finite sets and bijections.

\item 
The envelope $\env\Com$ is equivalent to $\Fin$, the category of finite sets.
The monoidal structure in the sum, which unit is the empty set.
The canonical morphism of operad $\opdO_1\to \Com$ induces the canonical inclusion $\FinBij\to \Fin$ between the envelopes.

\item 
The envelope $\env\Comnu$ is equivalent to $\FinSurj$, the category of finite sets and surjections.
The monoidal structure in the restriction of the sum along the inclusion $\FinSurj\to\Fin$.
The canonical morphism of operads $\Comnu\to \Com$ induces the canonical inclusion $\FinSurj\to \Fin$ between the envelopes.

\item 
The envelope $\env\Ezero$ is equivalent to $\FinInj$, the category of finite sets and injections.
The monoidal structure is the restriction of the sum along the inclusion $\FinInj\to\Fin$.
The canonical morphism of operads $\Ezero\to \Com$ induces the canonical inclusion $\FinInj\to \Fin$ between the envelopes.

\item 
The envelope $\env\As$ is equivalent to the category $\FinOrd$ whose objects are finite sets and morphisms are maps $A\to B$ enhanced with a total order on each of their fibers.
The monoidal structure in essentially the sum of finite sets.
The unit is the empty set.
The canonical morphism of operads $\As\to \Com$ induces the obvious functor $\FinOrd\to \Fin$ forgetting the order on fibers.
Since there exists a unique order on a singleton or the emptyset, there exists a canonical functor $\FinInj\to \FinOrd$.
It is easy to see that it is induced by the canonical morphism of operads $\Ezero\to \As$.

\item 
The envelope $\env\Asnu$ is equivalent to the category $\FinSurjOrd$ whose objects are finite sets and morphisms are surjective maps $A\to B$ enhanced with a total order on each of their fibers.
The monoidal structure is the restriction of the previous one along the inclusion $\FinSurjOrd\to\FinOrd$.
The canonical morphism of operads $\Asnu\to \Comnu$ induces the obvious functor $\FinSurjOrd\to \FinSurj$ forgetting the order on fibers.

\item The envelope of the operad $\AsMod$ is detailled in \cite[Definition~4.2.16]{LurieJ:higa} (where it is the subcategory of active morphisms).

\end{enumerate}	
\end{ex}

For the next results, we fix an operad $\opdP$ with set of colours $A$, and we consider $\monad\opdP \co \freerigA\to \freerigA$ the 2-rig monad of \cref{rmk:opd=monad}.
Recall the construction $\freesmcA_{\monad\opdP}$ from \cref{lem:monad-presheaves-monoidal}.
The following result provides alternative way to construct the monoidal envelope of $\Opd$.

\begin{lem}
\label{lem:equiv-Sigma=aS}
There exists an equivalence $\freesmcA_{\monad\opdP} = \envP$ of symmetric monoidal categories under $\freesmcA$.
\end{lem}
\begin{proof}
By construction, $\freesmcA_{\monad\opdP}$ and $\envP$ have the same objects which are those of $\freesmcA$,
\ie~sequences $\symelt am$ of objects of~$\catA$.
For the hom-set of maps between two sequences
$\vec a = \symelt am$ and $\vec a' = \symelt{a'}n$
we have
\begin{align*}
\hom{\freesmcA_{\monad\opdP}}{\vec a}{\vec a'}
& = \hom{\freerigA^{\monad\opdP}}{\monad\opdP(\vec a)}{\monad\opdP(\vec a')}\\
& \cong \hom\freerigA{\vec a}{\monad\opdP(\vec a')} \\
& \cong \hom\freerigA{\vec a}{\monad\opdP(\angles{a'_1})\otimes\dots\otimes \monad\opdP(\angles{a'_n})} \\
& \cong \HOM\freerigA{\vec a}{\matrice\opdP-{a'_1} \otimes\dots\otimes \matrice\opdP-{a'_n}}\\
& \cong \Big(\matrice\opdP-{a'_1}\otimes\dots\otimes\matrice\opdP-{a'_n}\Big)(\vec a)\\
& = \int^{\vec a''_1, \ldots, \vec a''_n \in \freesmcA} \HOM\freesmcA{\vec a}{\bigotimes_{j =1, \ldots, n}  \vec a''_j} 
\times \prod_{j=1, \ldots, n}  \matrice\opdP{\vec a''_j}{a'_j}   \,,  
\end{align*}
where used that $\monad\opdP$ is symmetric strong monoidal and the definition of the convolution
monoidal structure. This can be shown to coincide with the hom-sets of $\envP$, as
defined in~\eqref{eq:env} via a calculation that is left to readers.
\end{proof} 

Putting together \cref{lem:equiv-Sigma=aS,lem:monad-presheaves-monoidal} we get the following result.

\begin{prop}
\label{prop:em-for-freerig}
With the previous notations, there exist equivalences of 2-rigs under $\freerigA$
\[
\freerigA^{\monad\opdP}
\ \simeq\
\psh{\freesmcA_{\monad\opdP}}
\ \simeq\ 
\opdrigP
\,.
\]
\end{prop} 

\begin{rmk}
\label{rmk:naturality}
The equivalences of \cref{lem:equiv-Sigma=aS,prop:em-for-freerig} can be shown to be functorial on the 1-category $\Opdone$ of \cref{rmk:1-cat-opd} but not on the bicategory $\Opd$ because of the explicit use of the set of objects of $\opdP$ (see \cref{rmk:set-of-colours}).
\end{rmk}

A morphism of operads $f:\opdP\to \opdQ$ induces a symmetric strong monoidal functor $\env f:\envP\to \envQ$.
We shall need the following lemma for some computations in \cref{ex:adj-induction-restriction}.

\begin{lem}
\label{lem:monoidal-Kan-ext}
Let $\rigR$ be a 2-rig.
The left Kan extension of a symmetric strong monoidal functor $g:\envP\to \rigR$
along $\env f:\envP\to \envQ$ has a canonical symmetric strong monoidal structure.
\end{lem}
\begin{proof}[Proof (sketch)]
For $\vec b=\symelt bn$, an object of $\env Q$, we denote by $\env f\slice{\vec b} \defeq \envP \times_{\envQ}\envQ/\vec b$ the comma category of $\env f$ at $\vec b$.
The value at $\vec b$ of the left Kan extension $\mathrm{Lan}_f(g)$ of $g$ along $\env f$ is given by the colimit of the canonical diagram $\env f\slice{\vec b} \to \envP \to \rigR$, sending a pair $(\vec a,\alpha:f(\vec a)\to \vec b)$ to $g(\vec a)$.
Using \cref{eq:env}, we leave to the reader to verify that there exists a canonical equivalence 
$\env f\slice{\vec b} \simeq \prod_{i=1}^n\env f\slice{b_i}$.
Under this equivalence, the left Kan extension is
\begin{align*}
\mathrm{Lan}_f(g)(\vec b) 
&=\colim_{(\vec a,\alpha:f(\vec a)\to \vec b)\in \env f\slice{\vec b}}g(\vec a) \\
&=\colim_{(\vec a_i,\alpha_i:f(\vec a_i)\to b_i)_i\in \prod_i\env f\slice{b_i} }g(\vec a_1)\otimes \dots\otimes g(\vec a_n) \\
&=\bigotimes_i \Big(\colim_{(\vec a_i,\alpha_i:f(\vec a_i)\to b_i)\in \env f\slice{b_i} }g(\vec a_i)\Big) &&\text{by commutation of $\otimes$ with colimits in $\rigR$}\\
&= \mathrm{Lan}_f(g)(b_1)\otimes\dots\otimes \mathrm{Lan}_f(g)(b_n)\,.
\end{align*}
We leave also to the reader the proof that these canonical isomorphisms define a symmetric strong monoidal structure on $\mathrm{Lan}_f(g)$.
\end{proof}

\subsection{Operadic 2-rigs} 
\label{sec:opd-2rigs}

Consider the following pseudofunctor, sending an operad to the convolution 2-rig of its envelope:
\[
\begin{tikzcd}
\Opd \ar[r, "\Env"] & \SMCat \ar[r, "\Psh"] & \Rig \,.
\end{tikzcd}
\]

\begin{defn}
A 2-rig is said to be \emph{operadic} if it is equivalent to one of the form $\opdrigP$ for some operad~$\opdP$.
\end{defn}

\begin{ex}
\label{ex:opdrig}
The monoidal envelopes of \cref{ex:env} give the following examples of operadic 2-rigs (whose products are obtained by Day convolution):
\begin{enumerate}[label=(\alph*)]
\item When $\opdP=C$ is a category, 
its operadic 2-rig is simply $\opdrig{\opdO_C}\simeq\freerigC$ the free 2-rig on $C$.
In particular, when $C=1$, the envelope is the free 2-rig on one generator: $\opdrig{\opdO_1} \simeq \freerig1\simeq\psh\FinBij$.
\item $\opdrig\Com \simeq \psh\Fin$.
\item $\opdrig\Comnu \simeq \psh\FinSurj$.
\item $\opdrig\Ezero \simeq \psh\FinInj$.
\item $\opdrig\As \simeq \psh\FinOrd$.
\item $\opdrig\Asnu \simeq \psh\FinSurjOrd$.
\end{enumerate}	
\end{ex}

We write $\OpdRig$ for the full-bicategory of $\Rig$ spanned by operadic 2-rigs.
By definition, this bicategory fit into the image factorisation:
\[
\begin{tikzcd} 
\Opd \ar[r, two heads]
&  \OpdRig  \ar[r, hook]
& \Rig
\end{tikzcd}
\]

From the inclusions $\Set\to \Gpd\to \Cat\to \Opd$ and in~\eqref{eq:cat-smcat-2rig} we get also a diagram of inclusions:
\[
\begin{tikzcd}
\FsetRig \ar[r,hook]
&\FgpdRig \ar[r,hook]
&\FRig \ar[r,hook]
& \OpdRig \ar[r,hook]
& \CRig
\,.
\end{tikzcd}
\]

Recall that $\Rig$ and all its full sub-bicategories are tame bicategories (\cref{cor:RIG-is-tame,rmk:sub-tame=tame}).
By \cref{prop:EMC-embedding2}, the Eilenberg--Moore--Kleisli completions 
$\EMK(\FsetRig)$ and $\EMK(\FRig)$ 
are full sub-bicategories of $\Rig$ and we get a diagram of inclusions:
\[
\begin{tikzcd}
\FsetRig \ar[d,hook] \ar[r,hook]
&\FRig \ar[d,hook] \ar[rd,hook, bend left=20] 
\\
\EMK(\FsetRig)\ar[r,hook]
&\EMK(\FRig)
& \OpdRig
\,.
\end{tikzcd}
\]
We will show in \cref{thm:opdrig=em-frig} that the bicategories of the second row are all the same.
We need a lemma first.

\begin{lem}
\label{lem:opdrig=em-frig}
There exists an inclusion $\FRig \subseteq \EMK(\FsetRig)$.
\end{lem}
\begin{proof}
Let $\catC$ be a small category with set of objects $\catA$.
Then by \cref{lem:monad-presheaves}, the surjection $i \co \catA\to \catC$ induces a monadic adjunction $i_! \co \pshA\rightleftarrows \pshC \co  i^*$ in $\Pres$, and a cocontinuous monad $M=i^*i_!$ on $\pshA$.

Taking the image by the pseudofunctor $\Sym \co \Pres\to \Rig$, we get a 2-rig monad $M'$ on 
\[
\sym\pshA \simeq \freerigA \mathrlap{.}
\]
Recall also from \cref{lem:monad-presheaves} that $i_! \co \psh{A_0}\to \pshA$ is the Kleisli object of the monad $M$ in $\Pres$.
Since $\Sym$ is a left biadjoint, it preserves Kleisli objects.
Thus, the Kleisli object $\freerigA\to \freerigA_{M'}$ is $\sym{i_!} \co \sym\pshA \to \sym\pshC = \freerigC$. 
Since $\Rig$ is tame, the Eilenberg--Moore object of $M'$ is given by the right adjoint $\freerigC\to \sym\pshA$ to $\sym{i_!}$.

This proves that, for every small category $\catC$, the 2-rig $\freerigC$ belongs to $\EMK(\FsetRig)$.
\end{proof}

The following statement provides a universal property for the category of operadic rigs.

\begin{thm}
\label{thm:opdrig=em-frig}
The following equalities of sub-bicategories of $\Rig$ hold
\[
\OpdRig
\ =\ 
\EMK(\FsetRig)
\ =\ 
\EMK(\FgpdRig)
\ =\ 
\EMK(\FRig)\,.
\]
\end{thm}
\begin{proof}
We show the equality $\EMK(\FsetRig)=\EMK(\FRig)$ first.
The inclusion $\FsetRig\subseteq \FRig$ induces a fully faithful bifunctor 
$\EMK(\FsetRig)\subseteq \EMK(\FRig)$ by \cref{lem:EMC-embedding}.
The converse equality follows from \cref{lem:opdrig=em-frig,prop:EMC-embedding2}.
The equality with $\EMK(\FgpdRig)$ follows since $\FsetRig\subseteq\FgpdRig\subseteq\FRig$.

We will show that $\OpdRig=\EMK(\FsetRig)$ by showing that they are sub-bicategories spanned by the same objects.
The objects of $\OpdRig$ are 2-rigs of the type $\opdrigP$ for an operad $\opdP$.
The objects of $\EMK(\FsetRig)$ are monads $\opdQ \co \freerigA\to \freerigA$ in $\FsetRig$.
By \cref{rmk:opd=monad}, these are the same thing as operads $\opdQ$ coloured by the set $A$.
By construction of the embedding $\EMK(\FRig)\subseteq\Rig$ in \cref{prop:EMC-embedding2}, such a monad is send to the Eilenberg--Moore object $\freerigA^\opdQ$.
The equality $\OpdRig=\EMK(\FsetRig)$ is then equivalent to show that the 2-rigs of the type $\opdrigP$ and $\freerigA^\opdQ$ coincide.
But this is the statement of \cref{prop:em-for-freerig}.
\end{proof}

\begin{rmk} 
\label{rmk:Getzler}	
\Cref{thm:opdrig=em-frig} is related to Getzler's characterization of operads as 
regular patterns~\cite{Getzler2009}. The equality $\OpdRig=\EMK(\FgpdRig)$  is closely related to the approach of Kaufmann and Ward~\cite{Kaufmann-Ward} since one can characterize the objects of $\EMK(\FgpdRig)$ as the convolution 2-rigs over a Feynman category.
This characterization of operadic 2-rigs is generalized to the higher setting by Zetto \cite[Theorems 4.6]{Zetto}.
\end{rmk}

The next remark answers a question by Richard Garner.

\begin{rmk} 
\label{rmk:partial-monad}
The bicategory $\OpdBim$ can be regarded as a Kleisli bicategory for a relative pseudomonad in the sense of~\cite{FioreM:relpkbs}. Consider the diagram
\[
\begin{tikzcd}
& \Rig \ar[d,"\sfU"]
\\
\SMCat \ar[ur, bend left = 2, "\Psh"] \ar[r, hook] \ar[d, shift left = 2, "\END"]
& \SMCAT \ar[d, shift left = 2, "\END"]
\\
\Opd \ar[r, hook] \ar[u, shift left = 2, "\Env"] \ar[u, phantom, description, "\scriptstyle\dashv"] &  \ar[u, shift left = 2, "\Env"]  \OPD \,. \ar[u, phantom, description, "\scriptstyle\dashv"] 
\end{tikzcd}
\]

Here, we write $\Opd$ and $\OPD$ for the bicategories of small and large
operads, respectively. 
Similarly, $\SMCat$ and $\SMCAT$
denote the 2-categories of small and large
symmetric monoidal categories, respectively. 
Finally, $\RIG$ is the category of large 2-rigs.
The top triangle is the relative biadjunction \eqref{eq:adj-smcat-rig}. 
In the bottom square, the vertical maps form biadjunctions and the horizontal maps are inclusions. By composition, we obtain a relative biadjunction
\begin{equation}
\label{eq:relative-adj}
\begin{tikzcd}
 & \Rig \ar[d, "\END \circ \sfU"] \\
\Opd \ar[r, hook]  \ar[ur, bend left = 2, "\Psh \circ \Env"] & \OPD 
\end{tikzcd}
\end{equation}
which determines a relative pseudomonad $\END \circ \Psh \circ \Env \co \Opd \to \OPD$, whose Kleisli bicategory can be readily identified with $\OpdBim$,
as, for operads $\opdP$ and $\opdQ$, we have equivalences 
\begin{align*} 
\hom\OpdBim\opdP\opdQ & \simeq \Hom\Rig\opdrigP\opdrigQ\\
 & \simeq \Hom\SMCAT\envP\opdrigQ \\
 & \simeq \Hom\OPD\opdP{\End{\opdrigQ}}\,.
\end{align*}
\end{rmk}

\subsection{Operadic bimodules}

Our next result is the identification of the bicategory $\OpdRig$ with the bicategory $\OpdBim$ of operads and bimodules.
Recall from \cite[4.4]{GambinoN:opebaf} that this bicategory is defined as 
\[
\OpdBim
\ \defeq\ 
\Bim(\SetSym)\,.
\]

\begin{rmk}
\label{rmk:em-setsym=em-catsym}
Recall the biequivalences
$\SetSym \simeq \FsetRig$ and $\CatSym \simeq \FRig$ of \cref{prop:freerig-biequivalence} 
and that, for a tame bicategory $\ccatC$, $\EMK(\ccatC) \simeq \Bim(\ccatC)$ (\cref{sec:monad-tame}).
Then the biequivalence $\EMK(\FRig)\simeq \EMK(\FsetRig)$ of \cref{thm:opdrig=em-frig} recovers the equivalence $\Bim(\CatSym)\simeq \Bim(\SetSym)$ of \cite[Theorem 5.4.5]{GambinoN:opebaf}.
\end{rmk}

The next result follows from \cref{thm:opdrig=em-frig,rmk:em-setsym=em-catsym}.

\begin{thm}
\label{thm:opdrig=opdbim}
There exist biequivalences
\[
\OpdBim
\ =\ 
\Bim(\SetSym)
\ \simeq \ 
\Bim(\CatSym)
\ \simeq \ 
\OpdRig
\,.
\]
\end{thm}

\begin{rmk}
\label{rmk:opdrig=opdbim-2}
The resulting biequivalence $\Phi \co \OpdBim\to \OpdRig$ can be described explicitly as follows.
By the universal property of $\OpdBim$ and the fact that $\Rig$ is tame and 
Eilenberg--Moore--Kleisli complete, the tame pseudofunctor 
\[
\begin{tikzcd}[column sep = large]
\SetSym \ar[r, "\Psh \circ \mathsf{S}"] & \FsetRig \ar[r, hook] &  \Rig
\end{tikzcd}
\]
has an extension to a tame pseudofunctor $\Freerig^\dagger \co \OpdBim=\Bim(\SetSym)=\EMK(\SetSym) \to \Rig$ preserving Eilenberg--Moore--Kleisli objects, as in the diagram
\[
\begin{tikzcd}
\SetSym \ar[dr, bend right = 20, "\Freerig"'] \ar[r] & \OpdBim \ar[d, "\Freerig^\dagger"] \\
& \Rig \,. 
\end{tikzcd}
\]
The pseudofunctor $\Freerig^\dagger$ sends an operad $\opdP \co \freesmcA\op\times A\to \Set$ to the Eilenberg--Moore object of the associated monad $\monad\opdP:\freerigA\to \freerigA$.
By \cref{prop:em-for-freerig}, we have $\Freerig^\dagger(\opdP) = \freerigA^{\monad\opdP}=\opdrigP$.
At the level of morphisms, the functor $\Phi$ is given by composing 
the natural equivalence of \cref{lem:bim=k2em}
\[
\hom\OpdBim{\opdP}{\opdQ}
\ =\ 
\hom\Rig{\freerigA_{\monad\opdP}}{\freerigB^{\monad\opdQ}}\,,
\]
with the natural equivalence $\freerigA_{\monad\opdP} = \freerigA^{\monad\opdP}$ coming from the tameness of $\Rig$ (\cref{lem:em=k}).
The compatibility with composition is a long but safe computation.
\end{rmk}

\begin{cor}
\label{cor:opdrig=opdbim}
The functor $\Opd\to \OpdRig$ sends a morphism of operads to a left adjoint morphism of 2-rigs.
\end{cor}
\begin{proof}

Recall from \cite[Lemma 4.5.3]{GambinoN:opebaf}	that a morphism of operads $f \co (A,\opdP)\to (B,\opdQ)$ gives rise to an adjunction $f_\circ\dashv f^\circ$ in $\OpdBim$:
\begin{align*}
f_\circ \co \freesmcB\op \times A &\tto \Set
\\
(\symelt bn,a)	&\mto
\matrice\opdQ{b_1,\dots,b_n}{f(a)}
\\
f^\circ \co \freesmcA\op \times B &\tto
\Set
\\
(\symelt am,b)	&\mto \matrice\opdQ{f(a_1),\dots,f(a_m)}b\,.
\end{align*}
equipped with the obvious actions of $\opdP$ and $\opdQ$.
The construction of $F_\circ$ and $F^\circ$ from $F$ defines pseudofunctors 
$(-)_\circ \co \Opd \to \OpdBim$ and $(-)^\circ \co \Opd\op \to \OpdBim$.
Composing with the pseudofunctor $\Phi \co \OpdBim\to \OpdRig$ of \cref{rmk:opdrig=opdbim-2}, the bimodules $f_\circ$ and $f^\circ$ define respectively adjoint 2-rig morphisms 
\[
\begin{tikzcd}
\rigmor{f_\circ} \co \opdrigP
\ar[r, shift left=2]
\ar[r, phantom, description, "\scriptstyle{\vdash}" rotate=90]
&
\opdrigQ \co \rigmor{f^\circ} 
\ar[l, shift left=2]\,. 
\end{tikzcd}\qedhere
\]
\end{proof}

In particular when $\opdP=B$, the set of colors of $\opdQ$, and when $f$ is the canonical morphism $i \co B\to \opdQ$, we get an adjunction 
\begin{equation}
\label{eq:adj-colours}
\begin{tikzcd}
\rigmor{i_\circ} \co \freerigB
\ar[r, shift left=2]
\ar[r, phantom, description, "\scriptstyle{\vdash}" rotate=90]
&
\opdrigQ \co \rigmor{i^\circ} \,.
\ar[l, shift left=2]
\end{tikzcd}
\end{equation}

\begin{rmk}
\label{rmk:opdrig=opdbim}
One can show that the biequivalence $\Opd\Bim \simeq \OpdRig$  of \cref{thm:opdrig=em-frig} can be enhanced into a commutative square of pseudofunctors
\[
\begin{tikzcd}
\Opd \ar[d,two heads,"(-)_\circ"'] \ar[rr, two heads]
&& \OpdRig\ar[d, hook]
\\
\OpdBim \ar[r,"\simeq"]
&
\EMK(\FsetRig) \ar[r,hook]
&\Rig\,.
\end{tikzcd}
\]
where the functor $\Opd\to \OpdBim$ on the left is that of \cref{cor:opdrig=opdbim}.
The equivalence of commutation of the square is given on objects by \cref{prop:em-for-freerig}, 
its funtoriality with respect to morphismes of operads can be deduced from \cref{rmk:naturality},
but the naturality with respect to 2-cells requires a bit more work.
\end{rmk}

\begin{rmk} 
The bicategory $\Bim(\CatSym)  \simeq  \EMK(\FRig)$ can be given an explicit description in terms of the (symmetric) substitudes.
In keeping with \cref{rmk:set-of-colours,rmk:opd=monad} and the definition of~$\OpdBim$, a substitude can be defined as a monad in~$\CatSym$ and~$\Bim(\CatSym)$ can be thought of as the bicategory of substitudes and bimodules between them.
However, 
the biequivalences 
\[
\Bim(\CatSym) \simeq  \Bim(\SetSym)=\OpdBim 
\]
discussed in~\cref{rmk:em-setsym=em-catsym} show that we do not gain any new 2-rigs by considering substitudes instead of operads.
\end{rmk}

\subsection{Algebras over operads}

For a 2-rig $\rigR$, we have seen in \cref{rmk:partial-monad} that it was possible to define an endomorphism operad $\End\rigR$ for $\rigR$ but in the category $\OPD$ of large operads.

\begin{defn}
\label{def:alg}
For a (small) operad $\opdP$ and a 2-rig $\rigR$ we define the category of \emph{$\opdP$-algebras in $\rigR$} to be 
\[
\Alg\rigR\opdP
\ \defeq\ 
\hom\OPD\opdP{\End\rigR}\,.
\]	
\end{defn}

\begin{rmk}
\label{rmk:alg-explicit}
More explicitly, if $A$ is the set of colours of $\opdP$, an $\opdP$-algebras in $\rigR$ is a diagram $X:A\to \rigR$ together with maps
\[
\matrice\opdP{a'_1,\dots, a'_n}a \stto
\Hom\rigR{X(a'_1)\otimes \dots\otimes X(a'_n)}{X(a)}
\]
for every $a'_1,\dots, a'_n,a$ in $A$, which are compatible with the composition of operations.
\end{rmk}

If $f:\opdP\to \opdQ$ is a morphism of operads, it induces a functor $f^*:\Alg\rigR\opdQ\to \Alg\rigR\opdP$.
If $u:\rigR\to \rigS$ is a morphism of 2-rigs, it induces a functor $u_\dagger:\Alg\rigR\opdP\to \Alg\rigS\opdP$. 
Combining these definitions, we obtain a pseudofunctor $\mathsf{Alg}:\Opd\sfop\times \Rig \to \CAT$.
If $\rigR$ is fixed, we get a pseudofunctor
\begin{align*}
\Algname\rigR:\Opd\sfop &\tto \CAT	\\
\opdP &\mto \Alg\rigR\opdP\,.
\end{align*}
And if $\opdP$ is fixed, we get a pseudofunctor
\begin{align*}
\AlgP:\Rig &\tto \CAT	\\
\rigR &\mto \Alg\rigR\opdP\,.
\end{align*}
In particular, when $\rigR=\opdrigP$, the unit $\eta_\opdP:\opdP\to \End{\sfU(\opdrigP)}$ defines a canonical element of the functor $\AlgP$, which is a universal $\opdP$-algebra in a 2-rig, as the next result will show.
We will also need the counit $\epsilon_\rigR:\env{\End\rigR}\to \rigR$.

\begin{lem}
\label{lem:univ-alg}
The pseudofunctor $\AlgP$ is representable by $(\opdrigP,\eta_\opdP)$:
\[
\Alg\rigR\opdP
\ \simeq\ 
\hom\Rig\opdrigP\rigR
\]
\end{lem}
\begin{proof}
The statement follows by adjunction:
\begin{align*}
\Alg\rigR\opdP
& =
\hom\OPD\opdP{\End\rigR}
\\
& \simeq
\hom\SMCAT{\envP}\rigR && \text{by the adjunction \eqref{eq:hermida-opd-smoncat}}
\\
& \simeq
\hom\Rig\opdrigP\rigR && \text{by the relative adjunction \eqref{eq:adj-smcat-rig}.}\qedhere
\end{align*}
\end{proof}

In analogy with topos theory, we shall say that $\eta_\opdP$ is the \emph{universal $\opdP$-algebra} and that $\opdrigP$ is the \emph{classifying 2-rig} of $\opdP$-algebras (or of the operad $\opdP$).
Given a $\opdP$-algebra $A$ in a rig $\rigR$, the corresponding morphisms of 2-rig $f_A:\opdrigP\to \rigR$ is called the \emph{classifying morphism} of $A$.

\begin{rmk}
Notice that there is a factorization $\eta_\opdP:\opdP \to \End{\envP} \to \End{\sfU(\opdrigP)}$, 
through the unit of the adjunction $\Env\dashv \END$.
In other terms, the universal $\opdP$-algebra in $\opdrigP$ leaves within $\envP\subseteq\opdrigP$.	
\end{rmk}

In other terms, \cref{lem:univ-alg} says that constructing a morphism from an operadic 2-rig $\opdrigP \to \rigR$ is equivalent to defining a $\opdP$-algebras in $\rigR$.
Let us review some examples.

\begin{ex}
\label{ex:univ-alg}
Following up on \cref{ex:opdrig} and \cref{rmk:alg-explicit}:
\begin{enumerate}[label=(\alph*)]
\item\label{ex:univ-alg:cat} If $C$ is a category viewed as an operad, a $C$-algebra in $\rigR$ is simply a diagram $C\to \rigR$.
The universal $C$-algebra is simply the canonical generating diagram $C\to \opdrig{\opdO_C}\simeq\freerigC$.
In particular, when $C=1$, the universal algebra in $\opdrig{\opdO_1}\simeq \freerig1=\psh\FinBij$ is the generator $\hat 1 = \hom\FinBij-1$.

\item \label{ex:univ-alg:com}
The universal $\Com$-algebra in $\opdrig\Com \simeq \psh\Fin$ is the object $\hat 1 = \hom\Fin-1$ represented by a terminal object 1 in $\Fin$, equipped with the unique maps $\hat n = \hat 1^{\otimes n}\to \hat 1$, for $n\geq0$.
Notice that the object $\hat 1$ is terminal in $\opdrig\Com$.
(This fact is only true for $\Com$, as can be verified in the other examples.)

\noindent 

Consider the canonical morphism of operads $\opdO_1\to \Com$.
By \cref{eq:adj-colours}, we get an adjunction 
$i_!:\freerig1\rightleftarrows \psh\Fin:i^*$,
where $i:\freesmc1\simeq\FinBij\subseteq \Fin$ is the canonical inclusion.
Since the universal algebra in $\opdrig\Com$ is the terminal objet, its image by the right adjoint $i^*$ is the terminal object of $\freerig 1$.
Recall that the terminal presheaf is always the colimit of the Yoneda embedding $\freesmc1\to \freerig1$.
If $\hat n$ denote the representable presheaves of $\freerig 1$, the discussion above shows that the image of the free commutative algebra by $i^*$ is the presheaf 
\[
\coprod_{n\geq 0} \hat 1^{\otimes n}/\mathfrak S_n
\cong
\coprod_{n\geq 0} \hat n/\mathfrak S_n\,,
\]
\ie the free commutative monoid on the object $\hat 1$ in $\freerig1$.
(Beware that, although the previous computation shows that $i^*i_!\hat 1$ is the free commutative monoid on $\hat 1$, the monad $i^*i_!$ on $\freerig1$ is not the free commutative monoid monad, as can be seen from the isomorphism $i^*i_!\hat 2 = i^*i_!\hat 1\otimes i^*i_!\hat 1$.)

\item \label{ex:univ-alg:comnu}
The universal $\Comnu$-algebra in $\opdrig\Comnu \simeq \psh\FinSurj$ is the object $\hat 1$ equipped with the unique maps $\hat n\to \hat 1$, for $n\geq1$.
Proceeding as in \ref{ex:univ-alg:com}, the image of this object under 
$i^*:\psh\FinSurj \to \freerig 1$ is $\coprod_{n\geq 1} \hat 1^{\otimes n}/\mathfrak S_n$,
the free non-unital commutative monoid on $\hat 1$ in $\freerig 1$.

\item 
With the previous notation, the universal $\Ezero$-algebra in $\opdrig\Ezero \simeq \psh\FinInj$ is the object $\hat 1$, represented by a singleton 1 in $\FinInj$, equipped with the unique map $\hat 0\to \hat 1$.
The image of this algebra by the 2-rig morphism $i^*:\opdrig\Ezero\to \freerig 1$ is the object $\hat 0 + \hat 1$ in $\freerig 1$
(notice that $\hat 0$ is the unit of the monoidal structure of $\freerig 1$).

\item The universal $\As$-algebra in $\opdrig\As \simeq \psh\FinOrd$ is the object $\hat 1$, represented by a singleton 1 in $\FinOrd$, and structured by all the $n!$ maps $\hat n\to \hat 1$, for $n\geq0$. (Notice that 1 is no longer terminal in $\FinOrd$.)
 The image of this algebra by the canonical 2-rig morphism $i^*:\opdrig\As\to \freerig 1$ is the object 
\[
\coprod_{n\geq 0} \hat 1^{\otimes n} = \coprod_{n\geq 0} \hat n\,,
\] 
\ie the free monoid on the object $\hat 1$ in $\freerig 1$.

\item Similarly, the universal $\Asnu$-algebra in $\opdrig\Asnu \simeq \psh\FinSurjOrd$ is the object represented by a singleton 1 in $\FinSurjOrd$ structured by all the $n!$ maps $n\to 1$, for $n\geq1$.
And its image by $i^*:\opdrig\Asnu\to \freerig 1$ is the object 
$\coprod_{n\geq 1} \hat 1^{\otimes n}$,
\ie the free non-unital monoid on $\hat 1$ in $\freerig 1$.
\end{enumerate}	
\end{ex}

Recall the notion of an essentially surjective morphism of operads from \cref{ex:opd}\,\ref{def:opd-surj}.

\begin{lem}
\label{lem:adj-relative-free-alg}
For an essentially surjective morphism of operads $f:\opdP\to \opdQ$ and a 2-rig $\rigR$,
the adjunction of \cref{cor:opdrig=opdbim}
\[
\begin{tikzcd}
\rigmor{f_\circ}:\opdrigP
\ar[r, shift left=2]
\ar[r, phantom, description, "\scriptstyle{\vdash}" rotate=90]
&
\opdrigQ:\rigmor{f^\circ} 
\ar[l, shift left=2]\,. 
\end{tikzcd}
\]
and the induced adjunction
\begin{equation}
\label{eq:adj-relative-free-alg}	
\begin{tikzcd}
\rigmor{f^\circ}^*:
\Alg\rigR\opdP
\ar[r, shift left=2]
\ar[r, phantom, description, "\scriptstyle{\vdash}" rotate=90]
&
\Alg\rigR\opdQ
:\rigmor{f_\circ}^*
\ar[l, shift left=2]\,. 
\end{tikzcd}
\end{equation}
are both monadic.
\end{lem}

\begin{proof}
If the functor $\opdP\to \opdQ$ is essentially surjective, then so is the functor $\env\opdP\to \env\opdQ$.
Then the first statement follows from \cref{lem:surj-small=monadic,thm:rig-em-complete}.

The hom pseudofunctor $\Hom\Rig-\rigR:\Rig\op\to \CAT$ sends bicolimits to bilimits.
In particular, it sends Kleisli objects to Eilenberg--Moore objects.
Since $\Rig$ is tame (\cref{cor:RIG-is-tame}) it sends also Eilenberg--Moore objects to Eilenberg--Moore objects.
This proves the second statement.
\end{proof}

The functors $\rigmor{f^\circ}^*\dashv\rigmor{f_\circ}^*$ is sometimes called the \emph{extension--restriction} adjunction, like in \cite{FresseB:modoof}.

\begin{ex}
\label{ex:adj-induction-restriction}
We fix a 1-rig $(\rigR,\otimes,\unit)$.
We give a few examples of adjunction 
$\rigmor{f^\circ}^*\dashv\rigmor{f_\circ}^*$.
For this, we apply freely formula \eqref{eq:induction}
from \cref{app:anasf} to compute $\rigmor{f^\circ}^*$, 
as well as \cref{lem:monoidal-Kan-ext}.
\begin{enumerate}
\item When \eqref{eq:adj-relative-free-alg}	is applied to the canonical morphism $i:\opdO_1\to \Com$, we get an adjunction 
\[
\begin{tikzcd}
F_\Com:\rigR
\ar[r, shift left=2]
\ar[r, phantom, description, "\scriptstyle{\vdash}" rotate=90]
&
\Alg\rigR\Com:U_\Com
\ar[l, shift left=2]\,,
\end{tikzcd}
\]
where $U_\Com$ is the forgetful functor and its left adjoint $F_\Com$ is the free commutative monoid functor. 
A computation shows that it is given by
\[
F_\Com(X)
\cong
\colim_{N:\FinBij} X^{\otimes N}
\cong
\int^{\underline n\in \freesmc 1} \Com(n)\otimes X^{\otimes n}
\,.
\]

\item[(i')]
More generally, when \eqref{eq:adj-relative-free-alg}	is applied to the canonical morphism of operads $\opdO_A\to \opdP$, we get an adjunction
\begin{equation}
\label{eq:adj-free-alg}
\begin{tikzcd}
F_\opdP:\fun A\rigR
\ar[r, shift left=2]
\ar[r, phantom, description, "\scriptstyle{\vdash}" rotate=90]
&
\Alg\rigR\opdP:U_\opdP
\ar[l, shift left=2]\,,
\end{tikzcd}
\end{equation}
where the right adjoint $U_\opdP$ is the functor sending a $\opdP$-algebra $\opdrigP\to \rigR$ to its \emph{underlying diagram of objects}
\[
A\stto
\freerigA \stto
\opdrigP \stto \rigR\,.
\]
The left adjoint $F_\opdP$ is the \emph{free $\opdP$-algebra} functor on such diagrams, given by
\[
F_\opdP(a\mapsto X_a)
\cong
\int^{\vec a}
\Smatrice \opdP aa
\otimes 
\big(X_{a_1}\otimes\dots\otimes X_{a_n}\big)
\]

\item Using the canonical morphism $\Comnu\to \Com$, we get an adjunction 
\[
\begin{tikzcd}
F_\eta:\Alg\rigR\Comnu
\ar[r, shift left=2]
\ar[r, phantom, description, "\scriptstyle{\vdash}" rotate=90]
&
\Alg\rigR\Com:U_\eta
\ar[l, shift left=2]\,,
\end{tikzcd}
\]
where the functor $U_\eta$ forgets the unit of the $\Com$-structure.
Its left adjoint $F_\eta$, adding a free unit, is given by
\[
F_\eta(M)
\cong
(\unit + M)
\cong
\int^n
\Com(n)
\underset{\Comnu}\otimes X^{\otimes n}
\]

\item Using the canonical morphism $i:\Ezero\to \Com$, we get an adjunction 
\[
\begin{tikzcd}
F_\mu:\Alg\rigR\Ezero
\ar[r, shift left=2]
\ar[r, phantom, description, "\scriptstyle{\vdash}" rotate=90]
&
\Alg\rigR\Com:U_\mu
\ar[l, shift left=2]\,,
\end{tikzcd}
\]
where the functor $U_\mu$ forgets the multiplication of the $\Com$-structure.
Its left adjoint $F_\mu$ is the free commutative monoid on a unital object $\unit\to X$, given by the colimit (in $\rigR$)
\[
F_\mu(\unit\to X)
\cong
\colim_{N\in\FinInj} X^N
=
\int^n
\Com(n)
\underset{\Ezero}\otimes X^{\otimes n}
\]

\item Using the canonical morphism $\mathsf{ab}:\As\to \Com$, we get an adjunction 
\[
\begin{tikzcd}
F_\mathsf{ab}:\Alg\rigR\As
\ar[r, shift left=2]
\ar[r, phantom, description, "\scriptstyle{\vdash}" rotate=90]
&
\Alg\rigR\Com:U_\mathsf{ab}
\ar[l, shift left=2]\,,
\end{tikzcd}
\]
where the functor $U_\mathsf{ab}$ forgets the commutativity structure.
Its left adjoint $F_\mathsf{ab}$ is the abelianisation functor of monoids, given by the colimit (in $\rigR$)
\[
F_\mathsf{ab}(M) \smto \colim_{N\in \FinOrd} M^N
=
\int^n
\Com(n)
\underset{\As}\otimes M^{\otimes n}
\]

\item Using the canonical morphism $i:\Ezero\to \As$, we get an adjunction 
\[
\begin{tikzcd}
F_\mu:\Alg\rigR\Ezero
\ar[r, shift left=2]
\ar[r, phantom, description, "\scriptstyle{\vdash}" rotate=90]
&
\Alg\rigR\As:U_\mu
\ar[l, shift left=2]\,,
\end{tikzcd}
\]
where the functor $U_\mu$ forgets the multiplication of the $\As$-structure.
Its left adjoint $F_\mu$ is the free monoid on a unital object $\unit\to X$, given by the colimit (in $\rigR$)
\[
F_\mu(\unit\to X) \smto \colim_{\set n\in\Delta_+^\mathrm{inj}} X^n
\ =\ 
\int^n
\As(n)
\underset{\Ezero}\otimes X^{\otimes n}
\]

\item Using the canonical morphism $i:\Asnu\to \As$, we get an adjunction 
\[
\begin{tikzcd}
F_\mu:\Alg\rigR\Ezero
\ar[r, shift left=2]
\ar[r, phantom, description, "\scriptstyle{\vdash}" rotate=90]
&
\Alg\rigR\As:U_\mu
\ar[l, shift left=2]\,,
\end{tikzcd}
\]
where the functor $U_\eta$ forgets the unit of the $\As$-structure.
Its left adjoint $F_\eta$ is the free monoid on a non-unital monoid $\unit\to X$, given by the colimit (in $\rigR$)
\[
F_\eta(M) \smto \unit + M = \colim_{\set n\in\Delta_+^\mathrm{surj}} M^n
=
\int^n
\As(n)
\underset{\Asnu}\otimes M^{\otimes n}
\]
The ordinal $[0]$ is almost terminal in $(\Delta_+^\mathrm{surj})\op$ (it is only missing a map from $[-1]$)
Consequently $\{[-1],[0]\}\to \Delta_+^\mathrm{surj}$ is cofinal. 

\item Using the canonical morphism $i:\opdO_1\to \As$, we get an adjunction 
\[
\begin{tikzcd}
F_\As:\rigR
\ar[r, shift left=2]
\ar[r, phantom, description, "\scriptstyle{\vdash}" rotate=90]
&
\Alg\rigR\As:U_\As
\ar[l, shift left=2]\,,
\end{tikzcd}
\]
where the functor $U_\As$ forgets the $\As$-structure.
Its left adjoint $F_\As$ is the free monoid functor, given by the colimit (in $\rigR$)
\[
F_\eta(M) \smto  \colim_{\set n\in\Delta_+} M^n
\ =\ 
\int^n
\As(n)\otimes M^{\otimes n}\,.
\]

\end{enumerate}
\end{ex}

\begin{rmk}
\label{rmk:free-alg}
Recall from \cref{lem:univ-alg} the universal $\opdP$-algebra $\eta_\opdP$ in $\opdrigP$.
We have seen empirically in \cref{ex:univ-alg} that the image of $\eta_\opdP$ by $i^*:\opdrigP\to \freerigA$ is always the free $\opdP$-algebra $F_\opdP(i_A)$ on the canonical diagram $i_A:A\to \freerigA$.
We can now prove this as a general fact.
Consider the adjunction \eqref{eq:adj-free-alg} for $\rigR=\freerigA$.
Then, by construction of $F_\opdP$, the diagram $i_A$ is send by $F_\opdP$ to the $\opdP$-algebra classified by $i^*:\opdrigP \to \freerigA$.
More generally, if $X:A\to \rigR$ is an arbitrary diagram and $\bar X:\freerigA\to \rigR$ the corresponding 2-rig morphism, the free $\opdP$-algebra on $X$ in $\rigR$ is simply the image of $F_\opdP(i_A)$ by the 2-rig morphism $\bar X$.
\end{rmk}

\begin{lem}
\label{lem:coc-alg}
For an operad $\opdP$ and a 2-rig $\rigR$, 
the category $\Alg\rigR\opdP$ is cocomplete.
\end{lem}
\begin{proof}
Consider the adjunction \eqref{eq:adj-free-alg}.
The category $\fun A\rigR$ is clearly cocomplete.
The adjunction is monadic by \cref{lem:adj-relative-free-alg}.
Moreover both functor preserves sifted, hence filtered, colimits by \cref{prop:rigs-local-sifted}.
Then $\Alg\rigR\opdP$ is cocomplete by 
\cite[10.3]{Gabriel-Ulmer}.
\end{proof}

Let $\SIFTCATCOC\subseteq\SIFTCAT$ be the full subcategory spanned by the categories which are cocomplete.

\begin{prop}
\label{prop:alg-pt}
The functor $\Algname\rigR$ of algebras admits the following factorization:
\[
\begin{tikzcd}
\Opd\sfop \ar[r,"\Algname\rigR"] \ar[d,"\Psh \circ \Env"']& \CAT\\
(\OpdRig)\sfop \ar[>->, r] & \SIFTCATCOC \ar[u, "U"']\,.
\end{tikzcd}
\]
\end{prop}
\begin{proof}
The factorization through $(\OpdRig)\sfop$ follows from \cref{lem:univ-alg}.
And the factorization through $\SIFTCATCOC$ follows from \cref{prop:rigs-local-sifted,lem:coc-alg}.
\end{proof}

\Cref{lem:opd-in-opdbim} below has been obtained in the higher categorical setting in~\cite[Theorem~4.6]{Zetto}. 

\begin{lem}
\label{lem:opd-in-opdbim}
The pseudofunctors $\Env:\Opd\to \SMCat$ and $\Psh:\SMCat\to \Rig$, and therefore their composition, are 2-fully faithful (\ie~fully faithful on hom categories).
In particular, for any two operads $\opdP$ and $\opdQ$, the functor
\[
\Hom\Opd\opdP\opdQ \stto \Hom\Rig\opdrigP\opdrigQ
\]
is fully faithful.
\end{lem}
\begin{proof}
This is a computation left to the reader.
The first statement relies on the fact that the unit $\opdP\to \End\envP$ is a fully faithful morphism of operads \cite[Lemma 5.3]{BataninM:regpsf}.
And the second one derives from the 2-fully faithfulness
of $\Psh:\Cat\to \COC$.
\end{proof}

\begin{rmk}
\label{rmk:adjoint-rig-morphism}
\Cref{lem:opd-in-opdbim} means that morphisms between operadic 2-rigs (\ie~operad bimodules thanks to \cref{thm:opdrig=opdbim}) provide a faithful extension of morphisms of operads.
With this in mind, \cref{prop:alg-pt} can be understood to mean that the categories $\Alg\rigR\opdP$ are not only natural with respect to operad morphisms, but also with respect the larger class of morphisms between operadic 2-rigs.
Typically, every operad morphism inherits a right adjoint in $\OpdRig$ (\cref{cor:opdrig=opdbim}) which is not in general associated to an operad morphism (see \cref{ex:free-alg}).

Replacing operads by 2-rigs allows to classifies the same objects of interest---the algebras over operads---but by a means that is more flexible.
For example, it is not possible in general to construct a free $\opdP$-algebra in an operad $\opdQ$, but this is always possible in the 2-rig $\opdrigQ$.
\end{rmk}

\begin{ex}
Examples of 2-rig morphisms not associated to operad morphisms.
\label{ex:free-alg}
\begin{enumerate}[label=(\alph*)]
\item\label{ex:free-alg:abel}
Specializing \eqref{eq:adj-relative-free-alg} to the canonical morphism of operads $u:\As\to \Com$ (see \cref{ex:opd}\,\ref{ex:opd:com})
, we get an adjunction
\[
\begin{tikzcd}
F_u:\Alg\rigR\As
\ar[r, shift left=2]
\ar[r, phantom, description, "\scriptstyle{\vdash}" rotate=90]
&
\Alg\rigR\Com:U_u
\ar[l, shift left=2]\,. 
\end{tikzcd}
\]
where $U_u$ is the functor forgetting the commutativity structure and $F_u$ is the abelianisation functor.
The functor $U_u$ is induced by the 2-rig morphism $u^*\defeq\opdrig u: \opdrig\As\to\opdrig\Com$.
The functor $F_u$ is induced by the left adjoint 
$u_!\defeq\opdrig\Com\to \opdrig\As$.
Notice that the functor $u_!$ cannot be of the type $v^*$ for some operad morphism $v:\Com\to \As$ since there are no such morphisms. 

\item Other examples are given by the functor $\Alg\rigR{\Com}\to \Alg\rigR{\Com}$ which sends a commutative monoid $M$ to $\mathsf S_n M = M^{\otimes n}/\mathfrak S_n$, the $n$-th symmetric power of $M$,
and the functor sending $M$ to $\freesmc M = \coprod_n M^{\otimes n}/\mathfrak S_n$, the free symmetric monoid on $M$.
Both of them are easily representable by 2-rig morphisms $\opdrig\Com\to \opdrig\Com$ but cannot be representable by operad morphisms $\Com\to\Com$.
\end{enumerate}	
\end{ex}

\section{Enriched operadic rigs}
\label{sec:enror}

Let $\rigV$ be a 2-rig.
In this section, we explain how to generalize our results to operads enriched over $\rigV$.

\medskip
We define a \emph{$\rigV$-rig} to be a 2-rig morphism $\rho:\rigV\to \rigR$.
By abuse, we shall sometimes denote a $\rigV$-rig simply by $\rigR$, leaving the structure map $\rho$ implicit.
Given two $\rigV$-rigs $\rho:\rigV\to \rigR$ and $\sigma:\rigV\to \rigS$, a \emph{morphism of $\rigV$-rig} is a 2-rig morphism $f:\rigR\to\rigS$ together with an isomorphism $\phi:f\rho\simeq \sigma$.
We denote $\VRig$ the bicategory of $\rigV$-rigs.

\begin{short-rmk}
Any $\rigV$-rig $\rho:\rigV\to \rigR$ is always enriched (and cotensored) over $\rigV$ \cite{Janelidze2001}.
This provides an alternative definition of  a $\rigV$-rig as a $\rigV$-enriched category
that is presentable and symmetric monoidal in the enriched sense (thus, in particular, 
$\rigV$-tensored). We prefer to use the definition above for simplicity.
\end{short-rmk}

\begin{short-lem}
\label{lem:VRIG-tame:1}
The bicategory $\VRig$ is tame and the forgetful pseudofunctor $\VRig\to \Rig$ is tame.
\end{short-lem}
\begin{proof}
By definition, the category of morphisms of $\rigV$-rigs from $r:\rigV\to \rigR$ to $s:\rigV\to \rigS$ is defined as the fiber product in $\CAT$:
\begin{equation}
\label{eq:VRIG-tame}
\begin{tikzcd}
{\hom\VRig\rigR\rigS} \ar[r]\ar[d] & {\hom\Rig\rigR\rigS} \ar[d,"\precompo r"]\\
1\ar[r,"s"]	& {\hom\Rig\rigV\rigS}\,.
\end{tikzcd}
\end{equation}
Recall that a sifted category $I$ is always ``weakly contractible'' in the sense that $\colim_I 1 = 1$.
Therefore, any functor $1\to C$ always preserves sifted colimits.
This show that the cospan of \eqref{eq:VRIG-tame} is in $\SIFTCAT$.
Since the forgetful functor $\SIFTCAT\to \CAT$ creates limits, the square \eqref{eq:VRIG-tame} is therefore cartesian in $\SIFTCAT$.
This shows that $\hom\VRig\rigR\rigS$ has sifted colimits and that the functor $\hom\VRig\rigR\rigS \to\hom\Rig\rigR\rigS$ preserves them.
The compatibility of composition with sifted colimits follows.
Then the claim of the lemma can be deduced since reflective coequalizers are a particular case of a sifted colimit.
\end{proof}

\begin{short-lem}
\label{lem:VRIG-tame:2}
The forgetful pseudofunctor $U:\VRig\to \Rig$ has a left adjoint given by $\rigR \mapsto \rigV\otimes\rigR$.
Moreover, this is a tame pseudofunctor.
\end{short-lem}
\begin{proof}
This follows from the fact that $\rigV\otimes\rigR$ is the coproduct in $\Rig$.
Its tameness is left to the reader.
\end{proof}

\Cref{lem:VRIG-tame:1,lem:VRIG-tame:2} show that the adjunction $\rigV\otimes-:\Rig\rightleftarrows \VRig:U$ lives in the bicategory $\SIFTCAT$.

\begin{short-prop}
\label{prop:VRIG-tame}
The tame pseudofunctor $U:\VRig\to \Rig$ is pseudomonadic and creates Eilenberg--Moore objects.
\end{short-prop}
\begin{proof}
For any object $X$ in a bicategory $\ccatC$ with coproducts, the functor $Y\mapsto X+Y$ is a pseudomonad whose algebras are the maps $X\to Y$.
The creation of Eilenberg--Moore objects follows.
\end{proof}

Consider the 2-rig $\fun{\freesmcA\op}\rigV$ where the tensor product is given by Day convolution.
The functor of constant diagrams $\rigV\to \fun{\freesmcA\op}\rigV$ makes it into a $\rigV$-rig.

\begin{short-lem}
\label{lem:VRIG-EMK}
There exists a $\rigV$-rig equivalence $\fun{\freesmcA\op}\Set \otimes \rigV \cong \fun{\freesmcA\op}\rigV$.
\end{short-lem}
\begin{proof}
The equivalence of cocomplete categories $\fun{\freesmcA\op}\Set \otimes \rigV \cong \fun {\freesmcA\op}\rigV$ is a known formula for the tensor product of presentable categories.
The fact that it is an equivalence of 2-rigs follows by an explicit computation of the monoidal structure on each side.
\end{proof}

\begin{short-cor}
\label{cor:VRIG-EMK}
The forgetful pseudofunctor 	$U:\VRig\to \CAT$ has a partial left adjoint sending a small category $\catA$ to the $\rigV$-rig $\fun{\freesmcA\op}\rigV$.
\end{short-cor}
\begin{proof}
The partial left pseudoadjoint is obtained by composing $\rigV\otimes-$ with the partial left pseudoadjoint $\Cat\to \Rig$ of \cref{ex:rig}\,\ref{ex:rig:free}.
The description of the left pseudoadjoint follows by \cref{prop:VRIG-tame,lem:VRIG-EMK}.
\end{proof}

\begin{short-defn}[Free $\rigV$-rig]
We shall say that a $\rigV$-rig is \emph{free} (free on a set) if it is equivalent to a 2-rig $\fun \freesmcA \rigV$ for some small category $\catA$ (for $A$ a set).
We denote by $\FsetVRig\subseteq\FVRig\subseteq\VRig$ the full sub-bicategories spanned by free $\rigV$-rigs generated by a set or a category.
\end{short-defn}

Since $\VRig$ is tame by \cref{lem:VRIG-tame:1}, so is the full sub-bicategory $\FVRig$.
By \cref{prop:EMC-embedding2}, the Eilenberg--Moore--Kleisli completion of $\FVRig$ is a full subcategory spanned by Eilenberg--Moore objects of monads in $\FVRig$.

\begin{short-lem}
\label{lem:V-EMK}
The Eilenberg--Moore--Kleisli completions of $\FsetVRig$ and $\FVRig$ coincide.
\end{short-lem}
\begin{proof}
Since $\FsetVRig\subseteq\FVRig$, we need only to prove that $\FVRig$ is inside the Eilenberg--Moore--Kleisli completion of $\FsetVRig$.
The left pseudoadjoint $V\otimes-$ sends $\FRig$ to $\FVRig$ and Kleisli objects to Kleisli objects.
Using \cref{lem:opdrig=em-frig}, we get the expected inclusion.
\end{proof}

\begin{short-defn}[Operadic $\rigV$-rig]
We define the sub-bicategory of \emph{operadic $\rigV$-rigs} to be the full sub-bicategory $\OpdVRig\subseteq\VRig$ which is the Eilenberg--Moore--Kleisli completion of $\FsetVRig$ (or, by \cref{lem:V-EMK}, of $\FVRig$).
\end{short-defn}

\begin{short-prop}
\label{prop:V-opd}
The category $\OpdVRig$ is equivalent to the bicategory of $\rigV$-enriched operads and bimodules between them.
\end{short-prop}
\begin{proof}
By \cref{cor:VRIG-EMK}, we can compute the Eilenberg--Moore object of $\FVRig$ in $\Rig$.
Using \cref{lem:VRIG-tame:1}, we see that a monad in $\FVRig$ is a monad $M:\fun{\freesmcA\op}\rigV\to\fun{\freesmcA\op}\rigV$ in $\Rig$ together with coherent isomorphisms $M^n\circ r \cong r:\rigV\to\fun{\freesmcA\op}\rigV$, for all $n\geq 0$.
This unfolds to a symmetric sequence $\freesmcA\op\times A\to \rigV$ equipped with a monoid structure for the composition product.
When $A$ is a set, one can recognize the definition of a $\rigV$-operad with colours $A$.
Similarly, morphisms between such monads can be seen to corresponds to symmetric sequences $\freesmcB\op\times A\to \rigV$ with a bimodule structure over the two associated operads.
\end{proof}

\begin{short-rmk} 
\label{rmk:enriched-opd-morphism}
The classical bicategory of enriched operads and their morphisms can be constructed from $\OpdVRig$.
Let $\mathsf{Arr}(\OpdVRig)$ be the arrow bicategory of $\OpdVRig$ with the domain projection $\mathsf{dom}:\mathsf{Arr}(\OpdVRig) \to \OpdVRig$.
We consider the full sib-bicategory $\mathsf{Kl}(\OpdVRig) \subseteq \mathsf{Arr}(\OpdVRig)$ spanned by the \emph{Kleisli arrows}, that it the canonical $\rigV$-morphisms $\rigR\to \rigR^M$ associated to a monad $M$ in $\VRig$.
Consider the functor $\Set \to \FVRig \subseteq \OpdVRig$ sending a set to the corresponding free $\rigV$-rig.
The fiber product $\OpdVone$ of bicategories
\[
\begin{tikzcd}
\OpdVone \ar[r]\ar[d] & \Set \ar[d]\\
\mathsf{Kl}(\OpdVRig) \ar[r,"\mathsf{dom}"] & \OpdVRig 
\end{tikzcd}
\]
is the 1-category of $\rigV$-enriched operads an $\rigV$-enriched morphism between them.
To have the bicategory of such operads, we need two steps.
First, we need to replace $\Set$ by $\Cat$ and consider the fiber product
\[
\begin{tikzcd}
\mathsf{Subst}_\rigV \ar[r]\ar[d] & \Cat \ar[d]\\
\mathsf{Kl}(\OpdVRig) \ar[r,"\mathsf{dom}"] & \OpdVRig \,.
\end{tikzcd}
\]
The bicategory $\mathsf{Subst}_\rigV$ is the bicategory of \emph{$\rigV$-enriched substitudes} (\cref{rmk:set-of-colours}).
An object in $\mathsf{Subst}_\rigV$ is a morphism of $\rigV$-rigs $\fun{\freesmcA\op}\rigV \to \fun{\freesmcA\op}\rigV^M$, where $\catA$ is a small category.
We have seen in \cref{prop:V-opd} that the $\rigV$-rig $\fun{\freesmcA\op}\rigV^M$ is associated to an $\rigV$-enriched operad $\opdP:\freesmcA\op\times A\to \rigV$.
Let $|\opdP|$ denote the underlying operad of $\opdP$ (obtained by composing $\opdP$ with the ``hom from the unit'' functor $\hom\rigV\unit-:\rigV\to \Set$).
There exists a canonical functor $\catA\to|\opdP|_1$ (where $|\opdP|_1$ is the category of unary operations of $|\opdP|$).
Then the bicategory of $\rigV$-enriched operads is equivalent to the full sub-bicategory $\OpdV\subseteq \mathsf{Subst}_\rigV$ spanned by the objects for which $\catA\to |\opdP|_1$ is an equivalence.
(A similar reconstruction is done from the symmetric monoidal envelopes in \cite[5.17]{BataninM:regpsf}.)

In the higher categorical setting, an analog of this result for operads colored by groupoids has been proved in \cite[Theorem 5.6 and Corollary 5.16]{Zetto}.
\end{short-rmk}

\begin{short-rmk}
This approach to $\rigV$-enriched operads goes around the notion of $\rigV$-enriched categories.
By an argument similar to \cref{prop:free-coc}, the bicategory of small $\rigV$-enriched categories and ``$\rigV$-matrices'' between them can be recovered as the Eilenberg--Moore--Kleisli completion of the tame bicategory of free $\rigV$-modules, \ie $\rigV$-modules of the type $\fun{A\op}\rigV$ for $\catA$ a small (\emph{not} $\rigV$-enriched) category.
The bicategory of $\rigV$-enriched categories and enriched functors can be constructed from this one by a fiber product similar to that of \cref{rmk:enriched-opd-morphism}.
\end{short-rmk}

\appendix
\section{Analytic series formalism}
\label{app:anasf}

This appendix proposes a terminology and notation to compute with morphisms between free and operadic 2-rigs.
We start with a useful convention to lighten notations, inspired by \emph{Einstein's convention} in tensor calculus (see also \cite{AnelM:cofree}).
\begin{enumerate}[label=\alph*)]
\item For a covariant diagram $X \co C\to \rigR$ into a 2-rig $\rigR$, we denote by $\covariant Xc$ its value at the object $c$ of $C$.
\item For a contravariant diagram $F \co C\op\to \rigR$, we use instead the notation $\contravariant Fc$.
\item The coend between two such functors $F\otimes_CX = \int^{c\in C} \contravariant Fc\otimes \covariant Xc$ is abbreviated $\contravariant Fc\otimes \covariant Xc$, without the integral sign.
The convention is that there is an implicit coend on every variable appearing both on contravariant and covariant position.
The domain of the coend should be clear from the context.
\item More generally, we shall denote the values of a functor $F \co C\op\times D\to \rigR$ by $\bivariant Fcd$.
\item The same coend convention applies to notation with multiple indices.
Given another functor $G \co D\op\times E\to \rigR$, the object
$\bivariant Gde \otimes \bivariant Fcd \otimes \covariant Xc$ is the value at $e$ of the double coend $G\otimes_DF\otimes_CX$.

\item The role of the Kronecker tensor is played by the hom functor $\hom Ccd$ of the category $C$ indexing the coend, for which will also use the more suggestive notation $\bivariant Ccd$.
The Yoneda lemma is equivalent to either canonical isomorphisms
$\contravariant Xc \xto\simeq \bivariant Ccd \otimes \contravariant Xd$ or
$\covariant Xd \xto\simeq \bivariant Ccd \otimes \covariant Xc$.

\end{enumerate}

\medskip
We introduce the following terminology for a 2-rig $\rigR$.
\begin{enumerate}[label=\alph*)]
\item A \emph{(2-rig) analytic function} in $\rigR$ is a 2-rig morphism into $\rigR$ (of which we shall leave the domain implicit).
\item An \emph{point} of $\rigR$ is a morphism from $\rigR$ (of which we shall leave the codomain implicit).
\item The \emph{evaluation} of an analytic function $p$ at the point $x$ is the composition $p(x)  \co = x\circ p$.
\[
\begin{tikzcd}
\bullet \ar[r,"p"]\ar[rr,"p(x)"', bend right]
& \rigR \ar[r,"x"]
& \bullet
\end{tikzcd}
\]
\item When $\rigR=\freerigA$ is free, 
\begin{enumerate}
\item the \emph{series} of an analytic function $p$ is the corresponding functor $\freesmcA\op\to \Set$, and
\item the \emph{coefficients} of the series of an analytic function $p$ is the collection of values $\Scontravariant pa$ of the series, and
\item the \emph{coordinates} of a point $x$ are the values $\covariant xa$ of the restriction of $x$ along $A\to \freerigA$.
\end{enumerate}
Then the evaluation of $p$ at $x$ is given by the coend (where $\vec a = \symelt an$)
\[
p(x) = \Scontravariant pa
\otimes
\big(\covariant x{a_1}\otimes\dots\otimes\covariant x{a_n}\big)\,.
\]
\item 
For a 2-rig morphism $f \co \freerigA\to\freerigB$, the \emph{coefficients} of $f$ are the values $\Sbivariant fba$ of the corresponding functor $\freesmcB\op\times A\to \Set$
(\ie the coordinates of its series, or the series of its coordinates).

\item If $\rigR=\opdrigP$ is an operadic rig, 
we shall denote $\bimssj p$ and $\bimssj x$ the underlying left and right modules of an analytic function and a point of $\rigR$.
\begin{enumerate}
\item The morphism $\bimssj p$ is equipped with a left action of $\opdP$:
\[
\opdP\circ \bimssj f \stto \bimssj f
\,.
\]
We shall say that the $\Scontravariant {\bimssj f}a$ with the action of $\opdP$ form the \emph{$\opdP$-series} of $f$.
\item The morphism $\bimssj x$ is equipped with a right action of $\opdP$:
\[
\bimssj f \circ \opdP \stto \bimssj f
\,.
\]
We shall say that the $\covariant {\bimssj x}a$ with the action of $\opdP$ are the \emph{$\opdP$-coordinates} of $x$.
\end{enumerate}

\item For $f \co \opdrigP\to\opdrigQ$ a morphism between operadic 2-rigs, we denote by 
$\bimssj f \co \freerigA\to \freerigB$ its underlying bimodule
\[
\opdQ \circ \bimssj f \circ \opdP \stto \bimssj f
\,.
\]
We shall say that the $\Sbivariant {\bimssj f}ba$ form the \emph{$\opdQ$-series of $\opdP$-algebras} of $f$.

\item 
If $g \co \opdrigQ\to\rigR$ is an other morphism,
the right $\opdP$-module $\bimssj {g\circ f}$ is given by the reflective coequalizer
\[
\begin{tikzcd}
\bimssj g \circ \opdQ \circ \bimssj f
\ar[r, shift right = 2]
\ar[from=r]
\ar[r, shift left = 2] &
\bimssj g \circ \bimssj f
\ar[r] &
\bimssj g \bimcomp \opdQ \bimssj f \,.
\end{tikzcd}
\]

\item 
A functor $F \co \Alg\rigR\opdQ\to\Alg\rigR\opdP$ is called an \emph{analytic functor} if $F\simeq f^*$ for a morphism of 2-rigs $f \co \opdrigP\to\opdrigQ$.

\end{enumerate}

\medskip
Using all these definitions, Einstein's convention can be used to write explicit formulas for many functors.

\begin{enumerate}
\item For two 2-rig morphisms $f \co \freerigA\to\freerigB$ and $g \co \freerigB\to\freerigC$, the coefficients of the composition $g\circ f$ are given by the evaluation of $f$ on the point $g$:
\[
\Sbivariant {\big(g\circ f\big)}ca =  \Sbivariant fba
\otimes
\Scontravariant{\big(\covariant g{b_1}\otimes\dots\otimes\covariant g{b_n}\big)}c
\,.
\]
where the Day convolution
\[
\Scontravariant{\big(\covariant g{b_1}\otimes\dots\otimes\covariant g{b_n}\big)}c
\ =\ 
\big(\Sbivariant g{c_1}{b_1}\otimes\dots\otimes\Sbivariant g{c_n}{b_n}\big)\otimes \Sbivariant\freesmcB c{\otimes \vec c_i}\,.
\]
We shall put $\SSbivariantotimes gcb \defeq \Scontravariant{\big(\covariant g{b_1}\otimes\dots\otimes\covariant g{b_n}\big)}c$
and abbreviate
\[
\Sbivariant {\big(g\circ f\big)}ca =
\Sbivariant fba
\otimes
\SSbivariantotimes gcb
\,.
\]
A triple composition is then given by 
\begin{align*}
\Sbivariant {\big(h\circ g\circ f\big)}da
&=\Sbivariant fba
\otimes
\Scontravariant{\big(\covariant g{b_1}\otimes\dots\otimes\covariant g{b_n}\big)}c
\otimes
\Scontravariant{\big(\covariant h{c_1}\otimes\dots\otimes\covariant h{c_m}\big)}d\\
&=
\Sbivariant fba
\otimes
\SSbivariantotimes gcb
\otimes
\SSbivariantotimes hdc
\end{align*}

\item If $y \co \freerigB\to \rigR$ the same convention applies
and the analytic functor $f^*$ associated to $f \co \freerigA\to\freerigB$ is
\begin{align}
\label{eq:analytic-free}
\begin{split}
\fun B\rigR & \tto \fun A\rigR\\
\covariant yb &\mto
\Sbivariant fba
\otimes
\Sbivariantotimes yb
\,.
\end{split}
\end{align}
The analogy with real analytic functions should be clear enough.
We are going to push the analogy into operadic 2-rigs.

\item Let $\rigR=\opdrigP$ be an operadic rig, equipped with a function $p$ and a point $x$.
\begin{enumerate}
\item The left action of $\opdP$ on the series of $\bimssj p$ is given
\[
\Scontravariant {\bimssj p}{a'}
\otimes
\SSbivariantotimes \opdP a{a'}
\stto \Scontravariant {\bimssj p}a
\,.
\]
\item The right action of $\opdP$ on the coordinates of $\bimssj x$ is given by
\[
\Sbivariant \opdP {a'}a
\otimes
\Scovariantotimes {\bimssj x}{a'}
\stto \covariant {\bimssj x}a
\,.
\]
\end{enumerate}

\item
For a morphism $f \co \opdrigP\to\opdrigQ$, the action of $\opdP$ and $\opdQ$ on $\bimssj f$ is given by 
\[
\SSbivariantotimes \opdP {a'}a
\otimes
\SSbivariantotimes {\bimssj f}{b'}{a'}
\otimes
\SSbivariantotimes \opdQ b{b'}
\stto \Sbivariant {\bimssj f}ba
\,.
\]

\item 
If $y \co \opdrigQ\to\rigR$ is a point of $\opdrigQ$,
the right $\opdP$-module $\bimssj {y\circ f}$ is given by the reflective coequalizer
\begin{align*}
\Sbivariant {\big(\bimssj y\bimcomp\opdQ \bimssj f\big)}da
=\Sbivariant {\bimssj f}ba
\underset{\opdQ}\otimes
\SSbivariantotimes {\bimssj y}c{b}
&\defeq \colim\left(
\begin{tikzcd}[sep = small,ampersand replacement=\&]
\Sbivariant {\bimssj f}ba
\otimes
\SSbivariantotimes \opdQ {b'}b
\otimes
\SSbivariantotimes {\bimssj y}c{b'}
\ar[r, shift right = 2]
\ar[from=r]
\ar[r, shift left = 2] 
\&
\Sbivariant {\bimssj f}ba
\otimes
\SSbivariantotimes {\bimssj y}cb
\end{tikzcd}
\right) \,.
\end{align*}
The analytic functor $f^*$ between algebras is given by
\begin{align}
\label{eq:analytic-opd}
\begin{split}
\Alg\rigR\opdQ & \tto \Alg\rigR\opdP\\
\covariant {\bimssj y}b &\mto
\Sbivariant {\bimssj f}ba
\underset{\opdQ}\otimes
\Sbivariantotimes {\bimssj y}b
\,.
\end{split}
\end{align}
(Notice that the coend in \eqref{eq:analytic-opd} is computed in $\rigR$, not in $\Alg\rigR\opdP$, where its components would not always make sense.)
Formula~\eqref{eq:analytic-opd} shows what an analytic functor between categories of algebras is:
starting with a $\opdQ$-algebra $y$, one can 
\begin{enumerate}
\item compute arbitrary tensor powers of the coordinates of $y$, 
\item assemble these powers into an object of $\rigR$ by using the action of $\opdQ$ to contract them with a $\opdQ$-series,
\item lift this object to a $\opdP$-algebra, by means of an action of $\opdP$ on the $\opdQ$-series.
\end{enumerate}

\item 
For a morphism of operads $u \co \opdP\to \opdQ$, 
the coefficients of the induced morphism of 2-rigs $\Phi(u_\circ) \co \opdrigQ\to \opdrigP$ of \cref{cor:opdrig=opdbim} are
\[
\Sbivariant{\bimssj {\Phi(u_\circ)}}ba = \Sbivariant\opdQ b{u(a)}
\]
and the image of a $\opdQ$-algebra $y$ is given by
\begin{align}
\label{eq:restriction}
\begin{split}
\Phi(u_\circ)  \co \Alg\rigR\opdQ & \tto \Alg\rigR\opdP \\
\covariant yb
&\mto 
\covariant y{u(a)}
= \Sbivariant\opdQ b{u(a)}
\underset{\opdQ}\otimes
\Sbivariantotimes yb
\,.
\end{split}
\end{align}
Similarly, the coefficients of the morphism $\Phi(u^\circ) \co \opdrigP\to \opdrigQ$ are
\[
\Sbivariant{\bimssj {\Phi(u^\circ)}}ab = \bivariant\opdQ {u(\vec a)}b\,,
\]
and the $u$-enveloping $\opdQ$-algebra of a $\opdP$-algebra $x$ is given by
\begin{align}
\label{eq:induction}
\begin{split}
\Phi(u^\circ) \co \Alg\rigR\opdP & \tto \Alg\rigR\opdQ \\
\covariant xa &\mto
\bivariant\opdQ {u(\vec a)}b
\underset{\opdP}\otimes
\Sbivariantotimes xa\,.
\end{split}
\end{align}
\end{enumerate}

\end{document}